\documentclass[12pt, A4]{article}
\usepackage{amsmath,amssymb,amsfonts,mathrsfs,hyperref,microtype, amsthm}
\usepackage{geometry}
\usepackage[framed,numbered]{matlab-prettifier}
\usepackage{stackengine}
\usepackage[pdftex]{graphicx}
\usepackage{eufrak}
\usepackage{graphicx,color}
\usepackage[T1]{fontenc}
\usepackage{rotating}
\usepackage{booktabs}
\usepackage{mathtools}
\usepackage{float}
\usepackage{breqn}
\usepackage{graphicx}
\usepackage{caption}
\usepackage{multirow}
\usepackage{authblk}
\usepackage{lscape}
\usepackage{comment}
\usepackage{siunitx}

\newtheorem{rem}{Remark}[section]
\newtheorem{mydef}{Definition}[section]
\newtheorem{lem}{Lemma}[section]
\newtheorem{prop}{Proposition}[section]
\newtheorem{cor}{Corollary}[section]

\newtheorem{ex}{Example}[section]

\newtheorem{conj}{Conjecture}[section]
\newtheorem{open}{Open Problem}[section]
\numberwithin{equation}{section}

\newcommand\stackrqarrow[2]{%
	\mathrel{\stackunder[1pt]{\stackon[1pt]{$\rightsquigarrow$}{$\scriptscriptstyle#1$}}{%
			$\scriptscriptstyle#2$}}}

\def \E{{\rm I\kern-0.16em E}}
\def\P{{\rm I\kern-0.16em P}}
\def\F{{\rm I\kern-0.16em F}}
\def\B{{\rm I\kern-0.16em B}}
\def\C{{\rm I\kern-0.46em C}}
\def\G{{\rm I\kern-0.50em G}}
\def\D{{\rm I\kern-0.50em D}}

\newcommand{\VV}{\mathbb{V}}

\newcommand{\R}{\mathbb{R}}
\newcommand{\N}{\mathbb{N}}
\newcommand{\Lbb}{\mathbb{L}}
\newcommand{\Hbb}{\mathbb{H}}
\newcommand{\KK}{\mathbb{K}}

\newcommand{\RR}{\mathfrak{R}}

\font\eka=cmex10
\usepackage{ae}

\usepackage{color}

\def\ind{\mathrel{\hbox{\rlap{%
\hbox to 7.5pt{\hrulefill}}\raise6.6pt\hbox{\eka\char'167}}}}
\begin{document}
\title{ \bf On algebraic Stein operators for Gaussian polynomials }
\author{Ehsan Azmoodeh 
	\thanks{Department of Mathematical Sciences,
		University of Liverpool, Liverpool L69 7ZL, United Kingdom.
		E-mail: \texttt{ehsan.azmoodeh@liverpool.ac.uk}}, \quad 
	Dario Gasbarra
	\thanks{University of Helsinki, Department of Mathematics and Statistics, B 314, FI-00014, Finland. E-mail: \texttt{dario.gasbarra@helsinki.fi}} \quad Robert E$.$ Gaunt
	\thanks{The University of Manchester, Department of Mathematics, Alan Turing Building 2.217, M13 9PL Manchester, United Kingdom. E-mail:\texttt{ robert.gaunt@manchester.ac.uk}}
}

\date{ }                     
\setcounter{Maxaffil}{0}
\renewcommand\Affilfont{\itshape\small}

\maketitle 
\abstract 
The first essential ingredient to build up Stein's method for a continuous target distribution is to identify a so-called \textit{Stein operator}, namely a linear differential operator with polynomial coefficients. In this paper, we introduce the notion of \textit{algebraic} Stein operators (see Definition \ref{def:algebraic-Stein-Operator}), and provide a novel algebraic method to find \emph{all} the algebraic Stein operators up to a given order and polynomial degree for a target random variable of the form  $Y=h(X)$, where $X=(X_1,\dots, X_d)$ has i.i.d$.$ standard Gaussian components and $h\in \mathbb{K}[X]$ is a polynomial with coefficients in the ring $\mathbb{K}$. Our approach links the existence of an algebraic Stein operator with \textit{null controllability} of a certain linear discrete system. A \texttt{MATLAB} code checks the null controllability up to  a given finite time  $T$ (the order of the differential operator), and provides all \textit{null control} sequences (polynomial coefficients of the differential operator) up to a given maximum degree $m$.  This is the first paper that connects Stein's method with computational algebra to find Stein operators for highly complex probability distributions, such as $H_{20}(X_1)$, where $H_p$ is the $p$-th Hermite polynomial. Some examples of Stein operators for $H_p(X_1)$, $p=3,4,5,6$, are gathered in the Appendix and many other examples are given in the Supplementary Information.  


\section{Introduction} 

We begin with the following definition that plays a pivotal role in our paper.
\begin{mydef}\label{def:PSO}
Let $Y$ be a (continuous) target random variable. We say that a linear differential operator $\mathcal{S} = \sum_{t=0}^{T} p_t \partial^t $ acting on a class $\mathcal{F}$ of functions is a \emph{polynomial Stein operator} for $Y$ if (a) $\mathcal{S}f\in L^1(Y)$, (b) $\mathbb{E} \left[ \mathcal{S} f(Y) \right] =0$ for all $f \in \mathcal{F}$, and (c) the coefficients of $\mathcal{S}$ are polynomial. By $PSO_{\mathcal{F}}(Y)$ we denote the set of all polynomial Stein operators, acting on a class $\mathcal{F}$ of functions, for the target random variable $Y$.
\end{mydef}
In the last decade, polynomial Stein operators have got a lot of attention due to their important  role in the Nourdin-Peccati Malliavian--Stein approach \cite{np09,n-p-book}. This powerful method not only provides a drastic simplification of the classical \textit{method of the moments/cumulants}, but also allows for quantification of many significant probabilistic limit theorems that were not possible before.
Historically, in 1972, Charles Stein \cite{stein} introduced a powerful
technique for estimating the error in Gaussian approximations.  Stein's method for Gaussian approximation rests on the following fundamental Gaussian integration by parts formula: for $X\sim N(0,1)$ a standard Gaussian random variable, 
\begin{align}\label{ibpformula}
  \mathbb{E}\bigl[   Xf(X)-\partial f(X) \bigr]=0
\end{align}
for all absolutely continuous functions $f:\mathbb{R}\rightarrow\mathbb{R}$ such that $\mathbb{E}  |\partial f(X)|<\infty$.  Here $\partial=\frac{\mathrm{d}}{\mathrm{d}x}$ is the usual differentiation operator.  This formula leads to the so-called Stein equation: 
\begin{equation} \label{normal equation} xf(x)-\partial f(x)=h(x)-\mathbb{E}[h(X)],
\end{equation} 
where the test function $h$ is real-valued. It is straightforward to verify that $f(x)=-\mathrm{e}^{x^2/2}\int_{-\infty}^x\{h(t)-\mathbb{E}[h(X)]\}\mathrm{e}^{-t^2/2}\,\mathrm{d}t$ solves (\ref{normal equation}), and bounds on the solution and its derivatives in terms of the test function $h$ and its derivatives are given in \cite{chen,dgv}.  Evaluating both sides of (\ref{normal equation}) at a random variable $W$ and taking expectations gives
\begin{equation} \label{expect} \mathbb{E}  [Wf(W)-\partial f(W)]=\mathbb{E}  [h(W)]-\mathbb{E}[h(X)].
\end{equation}
Thus, the problem of bounding the quantity $|\mathbb{E}  [h(W)]-\mathbb{E} [h(X)]|$ has been reduced to bounding the left-hand side of (\ref{expect}).  A detailed account of Stein's method for Gaussian approximation and some of its numerous applications throughout the mathematical sciences are given in the monograph \cite{stein2} and the books \cite{chen,n-p-book}.   

One of the advantages of Stein's method is that the above procedure can be extended to treat many other distributional approximations.  In adapting the method to a new continuous distribution, the first step is to find a suitable analogue of the integration by parts formula (\ref{ibpformula}).  For a target random variable $Y$, this amounts to seeking a {\it Stein operator} $\mathcal{S}$ acting on a class of functions $\mathcal{F}$ such that $\mathcal{S}f\in L^1(Y)$ and
\begin{align*}\mathbb{E}\bigl[ {\mathcal S} f(Y) \bigr]=0. \end{align*}
For continuous distributions, $\mathcal{S}$ is typically a differential operator, 
although some operators in the recent literature are integral or fractional \cite{ah19, Liu}.  As noted by \cite{barbour2}, for a given distribution, there are infinitely many Stein operators.  A common approach to sifting through the available options is to restrict to {\it $T$-th order polynomial Stein operators} (see Definition \ref{def:PSO}) of the form
\begin{align}\label{stein:operator}
 {\mathcal S} f(y) = \sum_{t=0}^T  p_t(y) \partial^t f(y)
\end{align}
in which the coefficients belongs to the polynomial ring $\mathbb{K}[y]$ in the single variable $y$. (In this paper, $\partial^0\equiv I$, the identity operator.  For ease of exposition, we abuse notation and write $y$ in place of $yI$.)  In addition to their utility in the Malliavin-Stein method, such Stein operators are amenable to the various coupling techniques used in the implementation of Stein's method to derive distributional approximations, and as such the vast majority of differential Stein operators used in the literature take this form; for an overview see \cite{gms19,ley}.


 The  density method \cite{stein2,ls13,ley,stein3} leads to first order Stein operators, which are polynomial if the log derivative of the density is a rational function; this approach provides first order polynomial Stein operators for, amongst others, target distributions which belong to the Pearson family \cite{schoutens} or which satisfy a diffusive assumption \cite{dobler beta,kusuotud}.  Another method that naturally leads to first order Stein operators is the generator method of \cite{barbour2,gotze}.  However, it is often necessary to consider higher order operators; second order polynomial operators are needed for the Laplace \cite{pike}, PRR \cite{pekoz} and variance-gamma \cite{gaunt vg} distributions, for example.  This is because the densities of these distributions satisfy second order differential equations with polynomial coefficients.  In recent years, techniques have been developed for obtaining Stein operators in increasingly complex settings, such as the iterated conditioning argument for deriving Stein operators for products of a quite general class of distributions \cite{gaunt-pn,gaunt-ngb,gms19,gms19b} and the Fourier/Malliavin calculus  approach used to obtain Stein operators for linear combinations of gamma random variables \cite{aaps19a,aaps19b}. 

However, there remain many important distributional limits for which Stein's method has not yet been adapted to.  One identified by \cite{peccati14} to be of particular importance 
are those of the form $P(X)$, where $X\sim N(0,1)$ and $P$ is a polynomial of degree strictly greater than 2.  
The case in which $P$ is the $p$-th Hermite polynomial, $H_p(x)=(-1)^p\mathrm{e}^{x^2/2}\partial^d(\mathrm{e}^{-x^2/2})$, is of particular interest, due to their fundamental role in Gaussian analysis and Malliavin calculus. More importantly, they appear as the natural target distributions in the classic asymptotic theory of $U$-statistics, see  \cite[Section 4.4]{k-b-Ustatistics} and \cite[Chapter 3]{lee-Ustatistics}. For example, it is well-known that by taking the kernel $\psi(x_1,\ldots,x_p)=x_1x_2 \ldots x_p$, the limit of the corresponding $U$-statistics is $H_p (X)$, see page 87 in the latter reference.
 Indeed, since the seminal paper \cite{np09}, the Malliavin-Stein method has been used to derive numerous quantitative limits theorems for a wide class of laws (see, for example, \cite{ev13,ev15}), and has been particularly successful in the case that the target distribution is an element of the first or second Wiener chaos (see, for example, \cite{ aaps19a,a-m-p-s,a-g-tetilla,azmooden,dp16,et15,  n-n-p-mixedGaussian, np09, npr10, np15, npoly12}), but little is known regarding convergence to targets from higher order chaoses.  
 In the Stein's method literature, the only contribution is the modest one of \cite{gaunt-34} in which elementary but involved manipulations were used to obtain fifth and third order polynomial Stein operators for the target distributions $H_3(X)$ and $H_4(X)$, respectively, which represents a first step towards the long term goal of using the Malliavin-Stein method to understand convergence towards elements of higher order Wiener chaoses. In this paper we make progress that far exceeds that of \cite{gaunt-34}. 

In this work, we consider the multivariate Gaussian polynomial
$Y=h(X)$, where $X=(X_1,\dots, X_d)$ has independent and identically distributed (i.i.d$.$) standard
Gaussian components and $h\in \mathbb{K}[X]=\mathbb{K}[X_1,\dots,X_d]$, $\mathbb{K}=\mathbb{Z}$, $\mathbb{Q}$, or $\mathbb{R}$. For {\it any} target random variable $Y=h(X)$ of the above form, we show, in the formalism of linear system theory, that in Gaussian space, 
finding  Stein operators of the type (\ref{stein:operator}) with polynomial coefficients $p_t(y)\in\mathbb{K}[Y]$ is equivalent to
solving a  null controllability problem.  

Controllability is a fundamental concept in mathematical control theory. Consider a (finite-dimensional) discrete-time linear time-invariant system $x_{k+1} = Ax_k + Bu_k$, $k\in\mathbb{N}$ where $x\in\mathbb{R}^n$ is the state, $u\in\mathbb{R}^m$ is the input, and $A$ and $B$ are matrices of suitable size.
The controllability
property of such a system refers to the fact that any initial state can be steered to any final state by choosing the input appropriately. This notion was introduced by \cite{Kalman60} and was extensively studied by Kalman himself and many others; see \cite{Kalman,k80,son98} and the survey \cite{k93}. Controllability has many applications in control theory
and systems theory, as well as industrial and chemical process control, reactor control, control of electric bulk power systems, aerospace engineering and quantum systems theory. 

The significance of formulating the problem of finding polynomial Stein operators for Gaussian polynomials as a null controllability lies in the fact that there the problem of determining the controllability of a discrete-time linear system is well understood in the literature. The classical Kalman's rank test $\mathrm{rank}[B\,AB\,\cdots\,A^{n-1}B]=n$ is a necessary and sufficient for the
controllability of a linear finite dimensional discrete system. An alternative characterisation, which is
sometimes called the Popov-Belevitch-Hautus (PBH) test, presents an equivalent condition in terms of the
eigenmodes of the system \cite{am07,henri-1}. In Remark \ref{rem:Null-Controllability1}, we present the algorithm used in this paper for finding a null control sequence of a discrete-time linear system that can be implemented efficiently in modern computer algebra packages.

In Definition \ref{def:algebraic-Stein-Operator}, we introduce the notion of \emph{algebraic} Stein operators, a subclass of the class of polynomial Stein operators, which in fact coincides with the class of polynomial Stein operators in the case $d=1$ (see Proposition \ref{prop:Algebraic_Stein_Operators}). For a given target $Y=h(X)$, the above procedure can be implemented as an algorithm in modern computer algebra packages that can find \emph{all} algebraic Stein operators for $Y=h(X)$ up to a given order $T$ and maximal degree of polynomial coefficients $m=\max_{0\leq t\leq T}\mathrm{deg}(p_t(y))$.  We provide an efficient \texttt{MATLAB} code that is available to the research community at {\tt https://github.com/gasbarra/Algebraic-Stein-Equations}.


In dimension $d=1$, we use the method to study
the Stein operators for targets $Y=H_p(X)$. The `simplest'  Stein operators 
for $p=1, \dots, 6$
 are displayed in Appendix \ref{appendixa}. The Supplementary Information contains many more examples.  For $p\geq7$, the numerical coefficients in the Stein operators become too large to present in this paper, but we give some useful summaries of these Stein operators in Table \ref{tablesummary}.
 It should be clear that, apart from the first few cases,
it would be very hard to derive these Stein operators 
solely by bare-hands heroic calculations.  This paper therefore constitutes not only the first connection between Stein's method and control theory, but also the first study that deeply connects Stein's method with computer algebra programming to find Stein operators for highly complex probability distributions, such as $H_{20}(X)$. 

In addition to being able to find Stein operators that are out of reach of existing methods, our connection between Stein's method and null controllability has a feature not seen before in the Stein's method literature.  As our algorithm finds all algebraic Stein operators for a given target distribution $Y=h(X)$ up to a given order $T$ and polynomial degree $m$, we have confidence (at least in the case $d=1$) as to what polynomial Stein operators for that distribution are `simplest' in the sense of having minimal order $T$ or polynomial degree $m$.  For example, by running our \texttt{MATLAB} code to obtain polynomial Stein operators for $H_3(X)$ and $H_4(X)$ we know (with the aid of Proposition \ref{prop:higher-coeff}) that the fifth and third order operators that \cite{gaunt-34} obtained (see (\ref{h3x}) and (\ref{h4x})) are `simplest' in the sense of having minimum possible maximal degree 2 and have the minimum possible order $T$ for such operators, which was not obvious from the analysis of \cite{gaunt-34}. The complexity of Stein operators for $H_p(X)$ only increases as $p$ increases (see Table \ref{tablesummary}), with the implication being that there is a significant increase in difficulty in applying the Malliavin-Stein method for target distributions $H_p(X)$ for $p\geq3$. Another insight we gain from the outputs of our algorithm is that for $p\geq3$ there is not a unique `simplest' Stein operator in that the Stein operator with smallest possible order $T$ (degree $m$) does not have smallest possible maximal polynomial degree $m$ (order $T$). It remains to be seen whether the Stein operators with minimal order $T$ or minimal degree $m$ are more tractable when used in the Malliavin-Stein method.
Finally, we stress that the method of converting the problem of finding Stein operators into a null controllability problem applies more generally than even the Gaussian polynomial setting detailed in this paper. We expect this to be explored in future research.

We end this introduction by acknowledging that the focus of this paper is on finding Stein operators for Gaussian polynomials, but for sake of a focused treatment, we neglect the important question as to whether the Stein operators we obtain are characterising. This is the subject of our recent paper \cite{a-g-g21}.

The rest of our paper is organised as follows.  In Section \ref{sec2}, we review some fundamental Malliavin calculus operators and formulas on $\mathbb{K}[x_1,\ldots,x_d]$ that are needed in the paper.  
Section \ref{sec3} constitutes the heart of the paper.  In Section \ref{sec3.1}, we introduce the notion of a Stein chain, and make a connection between the existence of a Stein chain and polynomial Stein operators for Gaussian polynomials.  We present a method for validating Stein chains in Section \ref{sec3.2}.  In Section \ref{sec:Null}, we formulate the problem of finding Stein operators for Gaussian polynomials as a null controllability problem.  A concise description of the implementation of the null controllability approach in \texttt{MATLAB} is given in the Supplementary Information (SI).  In Section \ref{sec4}, we focus on applications to the important case of Gaussian Hermite polynomials $H_p(X)$, $X\sim N(0,1)$.  We provide a detailed description of the highest order coefficient of polynomial Stein operators for $H_p(X)$ in Section \ref{sec:Higher-Coef}.  In Section \ref{sec4.4}, we summarise some interesting features of the Stein operators for $H_p(X)$ that are obtained by our \texttt{MATLAB} code; examples of the Stein operators are given Appendix \ref{appendixa} and the SI.    
Finally, an auxiliary lemma is stated and proved in Appendix \ref{appendixc}.

\vspace{3mm}

\noindent{\emph{Note on the class of functions $\mathcal{F}$}:} Consider the polynomial Stein operator $\mathcal{S}=\sum_{t=0}^Tp_t(y)\partial^t$, with $\max_{0\leq t\leq T}\mathrm{deg}(p_t(y))=m$, for the target random variable $Y$, supported on $I\subseteq\mathbb{R}$. In this paper, the class of functions on which our polynomial Stein operators act is the class $\mathcal{F}_{\mathcal{S},Y}$, which is defined to be the set of all functions $f\in C^T(I)$ such that $\mathbb{E}|Y^jf^{(t)}(Y)|<\infty$ for all $t=0,\ldots,T$ and $j=0,\ldots,m$.  To ease notation, we write $PSO(Y)$ as shorthand for $PSO_{\mathcal{F}_{\mathcal{S},Y}}(Y)$.  We do not claim that this is the largest class of functions on which the Stein operators of this paper act, but the class is large enough for practical purposes and guarantees that in all of our proofs the crucial property $\mathbb{E}[\mathcal{S}f(Y)]=0$, $\forall f\in\mathcal{F}$, holds. We also highlight that when the target random variable $Y$ admits all moments (this is the case for all Gaussian polynomials $Y=h(X)$),
 the class $\mathcal{F}$ of functions as described above contains the class $C^\infty_c (\mathbb{R})$ of infinitely differentiable functions with compact support. 

\section{ Malliavin operators on $\mathbb{K}[x_1,\ldots,x_d]$ }\label{sec2}
For the scope of  our paper, it is enough  to define Malliavin calculus operators on $\mathbb{K}[x_1,\ldots,x_d]$ (see below for definitions) algebraically,  without discussing their  functional analytic extensions.
For a  state-of-the-art  exposition of Malliavin calculus  in full generality
see \cite{n-p-book,n-n-book}. 

\subsection{Algebraic preliminaries}

We   denote by $\mathbb{K}[x_1,\ldots,x_d]$ the commutative ring of polynomials in the variables $x_1,\dots,x_d$ with coefficients 
 in  $\mathbb{K}$, equipped with the usual addition
 and multiplication.
\begin{mydef}[Ideal]
 A subset $J\subseteq \mathbb{K}[x_1,\ldots,x_d]$ is
	an \emph{ideal} when
	\begin{enumerate}
		\item    $g_1, g_2\in J \Longrightarrow ( g_1 + g_2) \in J$; 
		\item  $f\in \mathbb{K}[x_1,\ldots,x_d], \; g \in J \Longrightarrow fg  \in J$.
	\end{enumerate}
\end{mydef}
\begin{mydef}
	The \emph{ideal} generated by $A\subseteq \mathbb{K}[x_1,\ldots,x_d]$ is the smallest ideal of $\mathbb{K}[x_1,\ldots,x_d]$ containing $A$,  
	denoted by
	\begin{align*}
	\langle A\rangle = \bigcap_{ \mbox{ ideal } J \supset A }
	J = \biggl\{  \sum_{\ell=1}^n  f_{\ell} g_{\ell}  :  f_{\ell}\in \mathbb{K}[x_1,\ldots,x_d],\, g_{\ell}\in A,\, n\in \mathbb{N} \biggr\}.
	\end{align*}
	When  $J= \langle g_1,\dots,g_n \rangle$
	we say that the  ideal is finitely generated.
\end{mydef}
\begin{rem}\label{rem:PID}
The ring $\mathbb{K}[x_1]$ is a principal ideal domain, meaning that every ideal $I \subseteq \mathbb{K}[x_1]$ is principal (can be generated by one element). This useful fact gives a precise description of the structure of the highest order polynomial coefficients in Stein operators, see Proposition \ref{prop:higher-coeff}. 
\end{rem}

\if 
For a given target polynomial $h(x)\in \mathbb{K}[x]=\mathbb{K}[x_1,\dots, x_d]$, we introduce the quotient polynomial ring 
\begin{align*}
\RR=\frac{\mathbb{K}[ x,y ]}{ \langle y-h(x) \rangle} , \end{align*}  where $y$ and $h(x)$ are identified.
The  Malliavin derivative of  an $\RR$-element with the equivalence $q(X,Y)=_{\RR} q(X,h(X) )$ is defined as 
\begin{align}\label{eq:D-qutient-ring}
D q( X,Y) =& \partial_X q(X,Y) + \partial_Y q(X,Y) \partial h(X) 
\end{align}
independently of  the representation. 
\fi 
\subsection{One-dimensional case}
The major application of our algebraic method is to present Stein operators for univariate Gaussian Hermite polynomials. We start with a purely algebraic presentation of Malliavin operators in dimension $d=1$.
 
\begin{mydef} In the univariate case, the Malliavin derivative  $D$, the divergence $\delta$ and its pseudo-inverse  
	$\delta^{-1}$ are defined as linear mappings acting  on the polynomial ring $\mathbb{K}[x]=\mathbb{K}[x_1]$, with
	\begin{align*} &D x^n=  \partial x^n = n x^{n-1}, \quad 
	\delta x^n =(x-\partial) x^n= x^{n+1}- n x^{n-1}, \quad \text{and}& \\ &
	\delta^{-1} 1 = 0 , \quad\delta^{-1} x= 1,\quad 
	\delta^{-1} x^n= x^{n-1} + (n-1) \delta^{-1} x^{n-2} = 
	\sum_{k=0}^{\lfloor \frac{n-1}2 \rfloor}
	\frac{ ( n-1){!}{!}}{ (n-1-2k) {!}{!} } x^{n-1-2k}, &  
	\end{align*}
	where $n{!}{!}$ denotes the double factorial. 
\end{mydef}
\begin{prop}
Let $n \in \mathbb{N}_0$. It holds that $	\delta^{-1} \delta x^n=x^n$. Moreover, 
\[ \delta \delta^{-1} x^n= x^n-\mathbb{E} \left[ X^n \right],\]
where $X\sim N(0,1)$
has moment sequence $\mathbb{E} \left[ X^{n} \right]= (n-1){!}{!}$ when $n\in 2\mathbb{N}$ and $\mathbb{E} \left[ X^{n} \right]=0$ otherwise. 
\end{prop}

\begin{lem}[Gaussian integration by parts] \label{lem:GIBP-d=1} Let $X \sim N(0,1)$. For $f,g\in \mathbb{K}[x]$,
	\begin{align*}
	\mathbb{E}\bigl[  f(X) D g(X) \bigr]= 	\mathbb{E}\bigl[  f(X) \partial  g(X) \bigr]= \mathbb{E}\bigl[  g(X) \delta f( X)  \bigr]. 
	\end{align*}
\end{lem}
\begin{proof}
By linearity, it suffices to consider the monomials $f(x)=x^n,$  $g(x)=x^m$.
	We have
	\begin{align*}
	\mathbb{E}\bigl[ X^n D X^m \bigr] = (m-1) \mathbb{E}\bigl[ X^{n+m-1} \bigr] = 
	\left \{ \begin{matrix}  (n+m-2){!}{!} (m-1),   &  (n+m)\mbox{ odd, }  \\
	0,  &  (n+m) \mbox{ even,  }  \\
	\end{matrix}
	\right . 
	\end{align*}
	which coincides with $\mathbb{E}\left[ X^m \delta X^n \right] = \mathbb{E} [ X^{m+n+1} ] -n \mathbb{E} [ X^{m+n-1} ].$ 
\end{proof}
\begin{mydef} We define the univariate Hermite polynomials as $H_0(x)=1$,  $H_1(x)=x$,  and  
	\begin{equation*} H_n(x) =\delta  H_{n-1}(x) = \bigl (\delta^n 1 \bigr) (x), \quad n \ge 2.
	\end{equation*} 
\end{mydef}
\begin{prop}[Properties of Hermite polynomials]
	\begin{enumerate}
		\item 
		$H_0(x),\dots, H_n(x)$ are monic polynomials spanning $\mathbb{K}_n[x]$ (the ring of polynomials of maximum degree $n$).
		\item 
		$\delta^{-1} H_n(x) = H_{n-1}(x) {\bf 1}(n>0)$.
		\item  Hermite polynomials are orthogonal in the sense that for $X \sim N(0,1)$ it holds that 
		\begin{equation*}
		\mathbb{E}\bigl[ H_m(X) H_n(X) \bigr]= 
		\mathbb{E}\bigl[ H_m(X) (\delta^{n} 1 ) (X) \bigr]=\mathbb{E}\bigl[ 1 \; \partial_n H_m(X) \bigr ] = \left \{
		\begin{matrix}  n!, &  n=m,  \\
		0,  &  \mbox{ otherwise. } \\
		\end{matrix} \right .
		\end{equation*}
		\item For $n \ge 1$, $D H_n(x)= \partial H_n (x)= n H_{n-1}(x)$.
		\item For every $0 \le m \le n$, and $X \sim N(0,1)$, 
		\begin{equation*} 
		  \mathbb{E}\bigl[ X^n H_m(X) \bigr] = {\bf 1}( (n-m)\in 2\mathbb{N} ) \frac{ n ! (n-m-1){!}{!} }{ (n-m)! }.
		\end{equation*}
	\end{enumerate}
\end{prop}
\begin{mydef} \label{def:Gamma_Y-dimension1}
	For a given (target) polynomial $y=h(x)$,  with $h\in \mathbb{K}[x]$, we introduce the linear operator $\Gamma_y$ acting on $\mathbb{K}[x]$ as 
	\begin{equation} \label{GammaY:1d} 
	\Gamma_y( f(x) ) = D h(x)  \delta^{-1}(f(x) ) =  \partial h(x) \delta^{-1}(f(x) ) \in \big \langle \partial h(x) \big\rangle.
	\end{equation}
\end{mydef}
\begin{lem}
Let $X \sim N(0,1)$. Consider $Y=h(X)$ where $h \in \KK[x]$. Assume that $f,g \in \mathbb{K}[x]$.	Then
	\begin{equation*}
	\mathbb{E}\bigl[ g( Y) f(X) \bigr] -\mathbb{E} [ g(Y)] \mathbb{E} [f(X) ] = \mathbb{E}\bigl[  \partial g(Y) \Gamma_Y( f(X) ) \bigr]. 
	\end{equation*}
\end{lem}

\begin{proof} Using Gaussian integration by parts formula Lemma \ref{lem:GIBP-d=1}, we can write 
\begin{align*}
	\mathbb{E}\bigl[ g( Y) f(X) \bigr] -\mathbb{E} [ g(Y)] \mathbb{E} [f(X) ] =\mathbb{E}\bigl[ g( h(X) ) \delta \delta^{-1}(  f(X) ) \bigr]  
		= \mathbb{E}\bigl[  \partial (g \circ h) (X)  \delta^{-1} f(X) \bigr]
\end{align*}	
\end{proof}
\subsection{Multidimensional case}
\begin{mydef}
	In the multivariate setting, we define  the Malliavin operators as follows. (a) The Malliavin derivative maps a polynomial $g\in \mathbb{K}[x_1,\dots,x_d]$ into its gradient:
	\begin{equation*}D_{\bullet} g = (D_k g)_{1\le k \le d} 
	 =  \nabla g =  ( \partial_{x_1} g(x_1,\ldots,x_d),\dots, \partial_{x_d} g(x_1,\ldots,x_d) ) \in \mathbb{K}[x_1,\dots, x_d]^d.
	\end{equation*}
(b)	The divergence $\delta$ maps in the opposite direction  a 
	polynomial vector $f(x)= f_{\bullet}(x_1,\ldots,x_d)=( f_1(x_1,\ldots,x_d),\dots, f_d(x_1,\ldots,x_d) )\in \mathbb{K}[x_1,\dots, x_d]^d$ into the polynomial
	\begin{equation*}
	\delta f_{\bullet}(x) = \sum_{k=1}^d \delta_k f_k(x) = \sum_{k=1}^d \bigl( f_k(x) x_k - \partial_{x_k} f_k( x) \bigr)  \in \mathbb{K}[x_1,\dots, x_d],
	\end{equation*}
	where $\delta_k$ denotes the univariate divergence operator operating on the $k$-th coordinate. 	
\end{mydef}

\begin{prop}
For $X \sim  N(0,I_d)$, a standard $d$-dimensional Gaussian random vector with covariance matrix the $d\times d$ identity matrix $I_d$, we have the multivariate Gaussian  integration by parts formula
	\begin{equation*}
	\mathbb{E}\big[  g(X) \delta f_{\bullet}(X) \big] = \mathbb{E}\Big[  \big(  D_{\bullet} g(X) , f_{\bullet}(X)  \big)_{\mathbb{R}^d} \Big].
	\end{equation*}
Here, and elsewhere, $(\cdot,\cdot)_{\mathbb{R}^d}$ denotes the usual inner product on Euclidean space $\mathbb{R}^d$.
\end{prop}
In order  to define a pseudo-inverse of $\delta$ in the multivariate setting, we use the multivariate Hermite polynomials instead of the monomial basis.
\begin{mydef} \label{def:Malliavin-Operators-Higher-Dimension}
	(a) The Ornstein-Uhlenbeck operator, which maps $\mathbb{K}[x_1,\dots, x_d]$ into itself, is defined as $L:= - \delta D$.\\
(b)	For a $d$-dimensional multi-index $\alpha = (\alpha_1,\ldots,\alpha_d) \in \mathbb{N}^d_0$, the multivariate Hermite  polynomials
	\begin{equation*} 
	H_{\alpha}(x) = H_\alpha (x_1,\ldots,x_d)= \prod_{k=1}^d  H_{\alpha_k}(x_k)
	\end{equation*}
	are eigenfunctions of $L$ with respective eigenvalues  $-|\alpha|:=-\sum_{k=1}^d \alpha_k$. 
	It also follows that $\mathbb{E}\bigl[ H_{\alpha}(X)H_{\beta}(X) \bigr]= {\bf 1}(\alpha =\beta) \prod_{k=1}^d\alpha_k !$.\\
(c)	The linear mapping $L^{-1}$  operates on the multivariate Hermite polynomials as
	\begin{align*}
	L^{-1} H_{\alpha} (x)= \left \{ \begin{matrix}  0, & \mbox{ when  } |\alpha|=0, \\
	-  |\alpha|^{ -1} H_{\alpha}(x), & \mbox{otherwise.}
	\end{matrix} \right. 
	\end{align*}
	(d) A pseudo-inverse of $\delta$ is defined as $\delta^{-1}= -DL^{-1}$, with $\delta^{-1} 1 =0$, and,  for $|\alpha|>0$,
	\begin{equation*}
	\delta^{-1} H_{\alpha}(x) = |\alpha|^{-1 } DH_{\alpha}(x) =\biggl( 
	\frac{\alpha_1}{|\alpha| } H_{\alpha- e_1}(x)  , \dots , \frac{\alpha_d}{|\alpha| } H_{\alpha-e_d}(x)  \biggr),
	\end{equation*}
	where as usual $(e_i : i=1,\ldots,d)$ denote the standard basis  for the Euclidean space $\mathbb{R}^d$. It follows by definition that  $\delta \delta^{-1} H_{\alpha}(x) = H_{\alpha}(x) {\bf 1}( |\alpha|>0)$. 
\end{mydef}
\begin{mydef} The $n$-th polynomial chaos in the variables $x_1,\dots, x_d$  is the linear subspace $\Hbb_n[x]= \Hbb_n[x_1,\ldots,x_d] \subseteq \mathbb{K}_n[x]= \mathbb{K}_n[x_1,\ldots,x_d]$ generated by the multivariate Hermite polynomials $H_{\alpha}(x)$ with $|\alpha|=n$, and the  decomposition
	\begin{equation*}
	\mathbb{K}_n[x]= \mathbb{K} + \Hbb_1[x] + \dots + \Hbb_n[x]
	\end{equation*}
 is  orthogonal with respect to the $d$-dimensional standard Gaussian measure.
\end{mydef}
\begin{mydef}\label{def:Gamma_Y-higher-dimension}
	In the multivariate case, with (target) polynomial $y=h(x_1,\dots,x_d)$, and $h \in \mathbb{K}[x_1,\ldots,x_d]$,  we define the linear \emph{Gamma operator} $\Gamma_y$ as:  $(x=(x_1,\ldots,x_d))$
 \begin{equation} \label{GammaY}
	\Gamma_y( f(x)  ) =  \bigl( D y, \delta^{-1} f(x) \bigr)_{\R^d}  = \bigl( Dh(x) , - DL^{-1} f(x)  \bigr)_{\R^d}.
	\end{equation}
	For example, $\Gamma_y(1)=0$ and for a multivariate  Hermite polynomial
	\begin{align*} 
	\Gamma_y( H_{\alpha}(x) ) = \sum_{k=1}^d   \frac{ \alpha_k}{|\alpha| } H_{\alpha-e_k}( x )\partial_{x_k} h(x).  
	\end{align*}
	Note that  $\Gamma_y$ maps $\mathbb{K}[x]$ into
	the ideal $
	\langle \nabla  h(x)\rangle=\big \langle \partial_{x_1} h(x),\dots ,  \partial_{x_d} h(x) \big \rangle $
	generated by the gradient of the target polynomial.
\end{mydef}
\begin{lem} For $X =(X_1,\ldots,X_d) \sim N(0,I_d)$, $Y=h(X)$, with $h \in \mathbb{K}[x_1,\ldots,x_d]$, and every $f\in \mathbb{K}[x_1,\ldots,x_d],g \in \mathbb{K}[Y]$,
	\begin{align}  \label{integration:by:parts:ii}
	\mathbb{E}\bigl[ g(Y) f(X) \bigr] 
	= \mathbb{E}[g(Y)]  \mathbb{E}[ f(X) ] +\mathbb{E}\bigl [\partial g(Y) \Gamma_Y( f(X)) \bigr]. 
	\end{align}
\end{lem}

\section{An algebraic formalism of Stein operators}\label{sec3}
Throughout this section, we adapt the following setting.
Let $d \ge 1$. Assume $X = (X_1,\ldots,X_d)$, where $X_1,\ldots,X_d$  are i.i.d$.$  standard Gaussian random variables.  Choose a polynomial $h \in \KK[x]=\KK[x_1,\ldots,x_d]$. We consider a  polynomial target random variable $Y$ in the Gaussian variates, namely that,  
\begin{equation}\label{eq:Polynomial-Target}
Y = h (X)=h(X_1,\ldots,X_d).
\end{equation}
Without loss of generality, we also assume that $\E[Y]=0$ (this assumption is always possible with a shift).   Furthermore,  the set $\KK[y] \cong \KK[h(x)]$ of all polynomials in the variable $y$ composed with the target polynomial  $h(x)$ is a subring of $\KK[x_1,\ldots,x_d]$. Denote by $H(Y)=L^2(\Omega,\sigma(Y),\gamma_d)$ the Hilbert space of all $Y$-measurable, square integrable random variables with respect to the $d$-dimensional standard Gaussian measure $\gamma_d$. Also, $H(Y)^{\perp}$ stands for the orthogonal complement of $H(Y)$.

\subsection{Forward Stein chain}\label{sec3.1}
We begin with the notion of a Stein chain. Our definition is quite abstract, but simply speaking, a Stein chain for the target random variable $Y$ is the sequence of polynomial coefficients of a Stein operator for $Y$ (when it exists); hence its name. This connection is made in Proposition \ref{prop:Chain-Leads-SteinOperator}. A benefit of this abstract definition is that it allows us to define the notation of an algebraic Stein chain in a natural manner. The distinction between general and algebraic Stein chains is important, because Stein operators whose polynomial coefficients form an algebraic Stein chain are easier to find; see Remark \ref{rem3.two}, item (i). 
\begin{mydef}[\textbf{Forward Stein chain}]\label{def:Stein-Chain}
Suppose that the target random variable $Y$ has the Gaussian polynomial form given by relation \eqref{eq:Polynomial-Target}. 
 \begin{itemize}
\item[(a)] A (general) \emph{Stein chain} of length $T \ge 1$ for $Y$ is a sequence $(p_t (y) \in \KK[y] : t=0,\ldots,T)$ such that: 
 $g_T + p_T (Y) \in H(Y)^{ \perp}$ where the sequence of Gaussian polynomials $g=(g_t \in \KK[X_1,\ldots,X_d]: t=0,\ldots,T)$ is recursively defined via 
\begin{align}
g_0 &=0, \nonumber \\
g_t &= \Gamma_Y (g_{t-1}+p_{t-1}(Y)), \quad t=1,\ldots,T, \label{eq:Stein-Cahin-Dynamic}
\end{align}
where the Malliavin Gamma mapping $\Gamma_Y$ is defined in  \eqref{GammaY:1d} and \eqref{GammaY} 
when $d =1$ and $d \ge 2$, respectively, and moreover $\E[g_t] = - \E[p_t(Y)]$ for $t=0,\ldots,T$. Let $SC(Y,T)$ denote the set of all Stein chains of length $T$, and let $SC(Y)= \bigcup_{T\ge 1} SC(Y,T)$ stand for the set of
all Stein chains of  finite length. 
\item[(b)] An algebraic Stein chain of length $T \ge1$ for  $Y$ is a Stein chain $(p_t (y) \in \KK[y] : t=0,\ldots,T)$ such that: $g_T + p_T (Y) \in \{ 0 \}$ almost surely. Let $ASC(Y,T)$ denote the set of all algebraic Stein chains of length $T$, and let $ASC(Y)= \bigcup_{T\ge 1} SC(Y,T)$ stand for the set of all algebraic Stein chains 
 of  finite length.
\end{itemize}
\end{mydef}

\begin{rem}\label{rem:Stein-Chain}
	\begin{enumerate}
		\item[(i)] The moment property $\E[g_t] = - \E[p_t(Y)]$ for $t=0,\ldots,T$ in Definition \ref{def:Stein-Chain} is immaterial, and can always be assumed due the fact that the linear  map $\Gamma_Y$ does not see constants, namely, $\Gamma_Y (f(X) +c) = \Gamma_Y (f(X))$ for any constant $c$. 
		\item[(ii)] We stress that the subring $\KK[Y]$ is not \emph{stable} under the linear mapping $\Gamma_Y$, meaning $\Gamma_Y (  \KK[Y]) \nsubseteq   \KK[Y]$. For example, a common candidate -- from the Malliavin-Stein perspective -- for the first polynomial entry in the Stein chain is $p_0 (Y) =cY$, $c\in \KK$. It can be readily seen that with the choices $Y=H_p (X)$, $p\ge 3$, and $p_0 (Y) = cY$ we have $g_1 \notin \KK[Y]$. Therefore, the most notable feature of a Stein chain or an algebraic Stein chain is that the final output of the linear machine \eqref{eq:Stein-Cahin-Dynamic}, either conditionally on the target variable $Y$ or without, is an element in the subring $\KK[Y]$.  
		\item[(iii)] Note that by the definition of the $\Gamma_Y$ operator, all the intermediate Gaussian polynomials $$g_t=g_t(X_1,\ldots,X_d)\in
		\big \langle \nabla h(X_1,\ldots,X_d) \big \rangle, \quad \forall \, t=0,\ldots,T.$$
		More importantly, when the Stein chain is algebraic,  $g_T \in \KK[Y]\cap \langle \nabla h(X_1,\ldots,X_d) \rangle$. On the other hand, $I= \KK[Y]\cap \langle \nabla h(X_1,\ldots,X_d) \rangle$ is an ideal of the subring $\KK[Y]$, and therefore it can be generated by one element (see Remark \ref{rem:PID}). In dimension $d=1$, for the targets of the form of a Gaussian Hermite polynomial, we present a prototypical element that generates  the corresponding ideal, see Proposition \ref{prop:higher-coeff} for details. 
	\end{enumerate}
\end{rem}

Let us continue with the following basic fact that tells us that the existence of a Stein chain of length $T$, in a natural way, leads to a polynomial Stein operator of order $T$ for the target random variable $Y$.

\begin{prop}\label{prop:Chain-Leads-SteinOperator}
Suppose that all the assumptions of Definition \ref{def:Stein-Chain} prevail. Then, the following statements hold: 
\begin{itemize}
	\item[(a)] For every $T \ge 1$, $ASC(Y,T) \subseteq SC (Y,T)$. Moreover, $ASC(Y) \subseteq SC(Y)$.
\item[(b)] Let $(p_t (y) \in \KK[y] \, : \, t=0,\ldots,T)$ be a forward Stein chain for the target random variable $Y$. Then $\mathcal{S} = \sum_{t=0}^{T} p_t \partial^t \in PSO(Y)$. Conversely, if $\mathcal{S} = \sum_{t=0}^{T} p_t \partial^t \in PSO(Y)$, then the sequence of polynomials $(p_t(y) \in \KK[y] \, :\, t =0,\ldots,T) \in SC (Y,T)$ is a forward Stein chain for $Y$ of length $T$.
\end{itemize} 
\end{prop}

\begin{proof} (a) This is clear because $0 \in H(Y)^{\perp}$. (b) Using iteratively the  integration by parts formula   \eqref{integration:by:parts:ii}, along with the moment property  $\E[g_t + p_t(Y)]=0$, for $t=0,\ldots,T$, in Definition \ref{def:Stein-Chain} (in particular, $\E[p_0(Y)] = g_0 = 0$), for every $f\in \mathcal{F}$, we can write 
	\begin{align*} &
	\E\biggl[  \bigl( p_0(Y)-\E[ p_0(Y) ]  \bigr) f(Y)  \biggr] = 
	\E\biggl[   \partial f(Y) \underbrace{ \Gamma_Y\bigl( p_0(Y)+g_0 \bigr) }_{ g_1} \biggr]=  \\ &
	-\E\bigr[  p_1(Y) \partial f(Y)   \bigr]
	+ \E\biggl[   \partial^2 f(Y) \underbrace{\Gamma_Y( p_1(Y)+g_1\bigr)  }_{ g_2 } \biggr] = \\ &
	-\E[  p_1(Y) \partial f(Y)   \bigr]
	-\E[ p_2(Y) \partial^2 f(Y)   \bigr] -
	\dots
	-\E[  p_{T-1}(Y) \partial^{T-1} f(Y)   \bigr]  \\ &
	+\E\biggl[   \partial^T f(Y) \underbrace{\Gamma_Y\bigl( p_{T-1}(Y)+g_{T-1} \bigr)  \bigr) }_{ g_T } \biggr]. 
	\end{align*} 
	Hence, $\E\bigl[\sum_{t=0}^T p_t(Y) \partial^t f(Y) \bigr] =\E[ p_0(Y) ]\E[f(Y) ]=0$,
	and therefore, the operator $\mathcal{S} = \sum_{t=0}^{T} p_t \partial^t$ is a polynomial Stein operator for the target random variable $Y$. For the other direction, let $\mathcal{S}=\sum_{t=0}^{T} p_t \partial^t \in PSO (Y)$ be a polynomial Stein operator for $Y$. Then the condition $\E[\mathcal{S}f (Y)]=0$ for all $f \in \mathcal{F}$ together with an iterative use of the  integration by parts formula yields that  $\E\left[ \partial^T f(Y) \left( p_T(Y) + g_T \right) \right]=0$ for every $f \in \mathcal{F}$. Hence, using a standard density argument we can conclude that $\E \left[ p_T (Y) + g_T \, \vert \, Y \right]=0$, and so $g_T + p_T (Y) \in H(Y)^{\perp}$. 
\end{proof}

Proposition \ref{prop:Chain-Leads-SteinOperator}, item (b), provides a neat correspondence between $PSO(Y)$ and $SC(Y)$. This useful fact in addition to the novel notion of an algebraic Stein chain introduced in Definition \ref{def:Stein-Chain}, item (b), directs us to the notion of an algebraic polynomial Stein operator that is the cornerstone of our paper.  
\begin{mydef}[\textbf{Algebraic polynomial Stein operator}]\label{def:algebraic-Stein-Operator}
	Let $d\ge 1$, and suppose that the target random variable $Y$ has the Gaussian polynomial form given by relation \eqref{eq:Polynomial-Target}. Let $T \ge 1$, and $\mathcal{S} = \sum_{t=0}^{T}p_t \partial^t \in PSO (Y)$. We say that $\mathcal{S}$ is an \textit{algebraic} polynomial Stein operator of length $T$ for target random variable $Y$ when the sequence of polynomials $(p_t(y) \in \KK[y] : t=0,\ldots,T) \in ASC(Y,T)$. We also denote by $APSO(Y)$ the class of all algebraic polynomial Stein operators for the target random variable $Y$.
	\end{mydef}

\begin{prop}\label{prop:Algebraic_Stein_Operators}
	Fix $d\ge 1$. Assume that the target random variable $Y$ has the Gaussian polynomial form given by relation \eqref{eq:Polynomial-Target}. Then, the following statements are in order:
	\begin{itemize}
		\item[(a)] $APSO(Y) \subseteq PSO(Y)$.
		\item[(b)] When $d=1$, we have $APSO(Y)=PSO(Y)$ (or equivalently $ASC(Y,T)=SC(Y,T)$ for every $T\ge 1$).
	\end{itemize}
	\end{prop}
\begin{proof} (a) Obvious.
(b) Let  $\mathcal{S} = \sum_{t=0}^{T}p_t \partial^t  \in PSO(Y)$. Then by definition, $$\E  \left[  g_T (X)  + p_T (Y)  \, \big \vert   \, Y  \right]= \E  \left[   g_T (X)  + p_T (h(X))  \, \big \vert   \, h(X)  \right]=0.$$
Then, Lemma \ref{lemma:conditional:expectation} implies that  either $g_T (X)+p_T(h(X))=0$ (that means $\mathcal{S} \in APSO(Y)$, and so nothing left to prove) or $h(X)$ is an even polynomial and  $g_T(X) + p_T (h(X))$ is  an odd polynomial. Next, first note that for $h$ being even, and any polynomial $p \in \KK[X]$, the composition polynomial $p(h(X))$ is  even. 
Next, we show that $g_T (X)=  \sum_{t=0}^{T-1}  \Gamma^{T-t}_Y p_t(h(X))$ is an even polynomial too.  Note that, linearity of the $\Gamma_Y$ operator allows us to write 
\begin{align}\label{eq:inside-component-symmetric}
 \sum_{t=0}^{T-1}  \Gamma^{T-t}_Y p_t(h(X))  =  \Gamma_Y  \Big(\sum_{t=1}^{T}  \Gamma_Y^{T-t}  \left(  p_{t-1} (h(X))    \right)      \Big).
\end{align}
Next, by definition of the $\Gamma_Y$ operator we can write $\Gamma_Y(g(X)) = h'(X) \delta^{-1} g(X)$ for every given polynomial $g\in\KK[X]$. Hence, the $\Gamma_Y$ operator maps an even polynomial $g(X)$ to a product of two odd polynomials, and $\Gamma_Y (g(X))$ is also even since $h(X)$ is an even polynomial. The above reasoning yields that the polynomial argument inside the $\Gamma_Y$ operator in the right hand side \eqref{eq:inside-component-symmetric}, which is given by  
\[ r(X): = p_{T-1}(h(X))  + \ldots+  \Gamma^{T-1}_Y (p_0 (h(X))),  \]
is even, 
and $\Gamma_Y (r (X))$ is  even as well.
Therefore, 
the polynomial $g_T(X) + p_T (h(X)) $, being both even  and odd,
 is necessarily zero, i.e$.$  
$\mathcal{S}  \in APSO(Y)$.  
\end{proof}

\begin{cor}\label{cor:Leading-Coff_Contain_h'(X)}
	Let $X \sim N(0,1)$, and  $Y = h(X)$ so that $h\in\KK[X]$ with $\deg{h} \ge 2$ and $\E[h(X)]=0$. Assume $\mathcal{S} = \sum_{t=0}^{T}p_t \partial^t \in PSO(Y)$ is a polynomial Stein operator for the target  $Y$. Identify $\KK[y] \cong \KK[h(x)]$. Then necessarily, $p_T (y) \in \KK[y] \cap \langle h'(x) \rangle$. 
\end{cor}

\begin{proof}
Proposition  \ref{prop:Algebraic_Stein_Operators}, item (b) implies that $\mathcal{S} = \sum_{t=0}^{T}p_t \partial^t \in APSO(Y)$. Hence $g_T(X) + p_T (h(X))=0$. On the other hand, note that 
\begin{align*} p_T (h(X))  &=  -  \Big(  \Gamma^T_Y (p_0) + \ldots+ \Gamma_Y (p_{T-1})  \Big) =  - \Gamma_Y   \Big(   \Gamma^{T-1}_Y (p_0) + \ldots+ p_{T-1}(X) \Big) \nonumber\\
& = -  \Big(  DY, \delta^{-1}  \left(  \Gamma^{T-1}_Y (p_0) + \ldots+ p_{T-1}(X)   \right)   \Big) \nonumber
\end{align*}
\begin{align*}
& =  -  h' (X)  \delta^{-1}  \left(  \Gamma^{T-1}_Y (p_0) + \ldots+ p_{T-1}(X)   \right)  \in \langle  h' (X)  \rangle.
\end{align*}
	\end{proof}
\begin{rem}\label{rem3.two}
	\begin{itemize}
\item[(i)] In several important settings (such as when $d=1$) algebraic polynomial Stein operators, when they exist, are easier to find compared to general polynomial Stein operators. This is due to the fact that one can settle the problem of finding an algebraic polynomial Stein operator in a completely algebraic framework -- using the algebraic operators that are introduced in Section \ref{sec2} -- to bypass the difficult problem of computing the conditional expectation in the last stage of the general Stein chain. In fact, hereafter, this is our major interest that at last successfully leads to many new polynomial Stein operators for Gaussian polynomial targets of very complex probabilistic nature, whose derivation is far beyond available techniques from the literature.     
\item[(ii)] For a Gaussian polynomial target random variable $Y$ of the form \eqref{eq:Polynomial-Target}, Proposition \ref{prop:Algebraic_Stein_Operators}, item (b), implies that 	$APSO(Y)=PSO(Y)$ when $d=1$. However, in general, 
\begin{equation}\label{eq:APSO=PSO}
APSO(Y) \neq PSO(Y),  \quad  \text{ when } \quad d\ge 2.
\end{equation}
 This phenomenon is discussed in the forthcoming example, in which the significant role of probability enters through the conditional expectation. As will become clear, in higher dimensions, with the target random variable $Y$ of the form \eqref{eq:Polynomial-Target}, the situation  $\E[g(X_1,\ldots,X_d)\, \vert\, Y ] = p(Y)$ most often happens for some polynomial $p\in\KK[Y]$, although $g(X_1,\ldots,X_d) \in \KK[X_1,\ldots,X_d] \setminus \KK[Y]$. Also, item (b) of the forthcoming example contains several examples of Gaussian polynomial target random variables $Y=h(X_1,\ldots,X_d)$ with $d\ge 2$ for which $APSO(Y) \neq \emptyset$. 
\end{itemize}
\end{rem}

\begin{ex}
\begin{itemize}
	\item[(a)] {\bf Case $d=1$:} Let $X \sim N(0,1)$. Although Proposition \ref{prop:Algebraic_Stein_Operators}, part (b), confirms that $APSO(Y)=PSO(Y)$, we would like to present some concrete illustrative examples for some known polynomial Stein operators.		
(i) When $Y=aX+b$, where $a,b\in\mathbb{R}$, it is clear that $APSO(Y)=PSO(Y)$.	(ii) Let $Y=aX^2+bX+c$, where $ a,b \neq 0$ and $a,b,c \in \R$. Note that $\E[Y]=a+c$. So,  from the Malliavin--Stein perspective, the most desirerable choice for the zeroth polynomial is $p_0 (y)=y - a-c$. Straightforward computations give that 
\begin{align*}
g_1(X)&= \Gamma_Y(p_0(Y))= \Gamma_Y(Y-a-c)=\Gamma_Y (Y)\\
&= 2a^2H_2(X)+3abH_1(X)+(2a^2+b^2),
\end{align*}
 and $\E[g_1(X)]=2a^2+b^2$. The requirement $a,b\neq 0$ yields that $g_1 (X) \neq p(Y)$ almost surely for any polynomial $p \in \KK[Y]$. Hence, we need to choose appropriately the next polynomial $p_1 (y) \in \KK[y]$ and move to the next stage. To do this, note that
\begin{align*}
\Gamma_Y (g_1 (X)) = 4a^3 H_2 (X) + 8a^2b H_1 (X) + (3ab^2+4a^3).
\end{align*}
We now aim to find a polynomial $p_1(y)$ that satisfies the requirement  $\E[g_1(X)]=-\E[p_1(Y)]$ and the terminating condition that $g_2=\Gamma_Y ( g_1 + p_1(Y)   )\in H(Y)^{ \perp}$. The polynomial $p_1(y) = -4a (y-c)+ (2a^2 -b^2)$ meets these requirements. It is easily checked that $\E[g_1(X)]=-\E[p_1(Y)]$ and we also have that
\begin{align*}
g_2(X)& =\Gamma_Y \left( g_1 (X)) + p_1(Y)   \right)  =  \Gamma_Y (g_1 (X) - 4a \Gamma_Y (Y)\\
&=\Gamma_Y (g_1 (X)) - 4a \Gamma_Y (p_0(Y))\\
&=\Gamma_Y (g_1 (X)) - 4a g_1 (X)\\
& = -4a^3H_2 (X)-4a^2bH_1(X) -4a^3 -ab^2\\
&= -4a^2 (Y-c) -ab^2.
\end{align*}
We now choose $p_2 (y)= 4a^2 (y-c)+ab^2$, and note that $g_2+p_2(y)=0$. Thus, the sequence $(p_0(Y),p_1(Y),p_2(Y))$ is an algebraic stein chain of length $T=2$ for the target random variable $Y$.
Finally, applying Proposition \ref{prop:Chain-Leads-SteinOperator} yields the Stein operator   	
\[ \mathcal{S} = \left( ab^2+4a^2(y-c) \right)\partial^2 + \left( 2a^2-b^2-4a(y-c) \right)\partial+ (y-c-a), \]
	which was achieved in \cite[Proposition 2.1]{gaunt-34} via a different approach.

	

	\item[(b)] {\bf General Case:} Let $\{ X_k \sim N(0,1) \, : \, k \in \N \}$ be a family of i.i.d$.$ standard Gaussian random variables. 
	
\vspace{2mm}
	
	(i) Consider the target random variables $	Y=\sum_{k=1}^{K} \alpha_k H_2 (X)$
	 (living in the second Wiener chaos) where the $\alpha_k$ are non-zero real numbers. Note that by choosing $\alpha_k =\alpha$ for $k=1,\ldots,K$, the random variable $Y$ boils down to that of the \textit{centered gamma} distribution with $K$ degrees of freedom, and by choosing $\alpha_k=\alpha,-\beta$ for $k=1,\ldots,r$, and $k=r+1,\ldots,2r (=K)$, respectively, where $\alpha  \beta > 0$, the random variable $Y$ boils down to that of the centered \textit{variance-gamma} distribution \cite{gaunt vg} with parameters $(r,\theta,\sigma)$, where $\theta=\alpha - \beta$ and $\sigma=2\sqrt{\alpha\beta}$.   In \cite{aaps19b}, the authors provided the following polynomial Stein operator of order $T$ with linear coefficients for the target random variable $Y$:
	\begin{equation}\label{eq:Stein_Operator_2Chaos}
	 \mathcal{S}  =  \sum_{k=1}^{T} (b_k - a_k y)\partial^k - a_0 y ,
	 \end{equation}
	where the constants $(a_k,b_k  : k=0,\ldots,T)$ are explicit and given in \cite[Section 2.1]{aaps19b}, and $T$ is the number of distinct coefficients in $(\alpha_k,\, k=1,\ldots,K)$.  Next, relying on \cite[Lemma 3]{azmooden}, relation (23), we can easily infer that $\mathcal{S} \in APSO(Y)$ is an algebraic polynomial Stein operator for $Y$. Up to now, only when $K=1$ (equivalently $d=1$, and corresponding to a scaled centered gamma random variable) do we know that $APSO(Y)=PSO(Y)$. For example, with the product Gaussian distribution $Y=X_1X_2$ we do not know whether $APSO(Y)=PSO(Y)$. 

\vspace{2mm}
	
	 (ii) The following example provides a polynomial Stein operator that is not algebraic with $d=2$. Consider the target $Y=X^2_1 X^2_2$, a product of two independent chi-square random variables, each with one degree of freedom. In \cite{gaunt-ngb}, the author provided the following second order polynomial Stein operator 
	\begin{equation}\label{eq:Stein_Operator_Product_Gamma}
	\mathcal{S}  = y^2 \partial^2 + 2y\partial +\frac{1}{4} (1-y) .
	\end{equation}
	Using some straightforward computations, one can see that 
\[g_1(X_1,X_2)= \Gamma_Y (p_0(Y)) = -\frac{1}{4} \Gamma_Y (Y) =  -\frac{ Y}{ 4} \biggl( H_2 (X_1) + H_2(X_2) + 4  \biggr)  \in \KK[X_1,X_2].\]
 Also, $g_2(X_1,X_2)$ contains a term in the eighth Wiener chaos of the form $H_6(X_1)H_2(X_2)$  that does not appear in the chaos expansion of $Y^2$. This yields that $g_2(X_1,X_2)+p_2(Y)=g_2(X_1,X_2)+Y^2 \neq 0 $ almost surely, and hence, $\mathcal{S}   \notin APSO(Y)$ is not an algebraic polynomial Stein operator for the target $Y$. It is an interesting problem to determine whether $APSO(Y) \neq \emptyset$. 
 
\vspace{2mm} 
 
	 (iii) The following example provides a polynomial Stein operator that is not algebraic with $d=3$. Let $Y=X_1X_2X_3$ be the product of three independent standard Gaussian random variables.  In \cite{gaunt-pn}, the author provides the following third order polynomial Stein operator for the target random variable $Y$:
	\[  \mathcal{S} = y^2 \partial^3 + 3y \partial^2 + \partial - y \in PSO(Y). \]
	Some tedious computations yield that 
	\begin{align*}
	g_1 (X_1,X_2,X_3) & = - \Gamma_Y (Y) \\
	&= -\frac{1}{3} \Big(  H_2 (X_2) H_2(X_3)+ H_2 (X_2)+ H_2(X_3)+1 \Big)\\
	& \quad  -  \frac{1}{3} \Big(  H_2 (X_1) H_2(X_3)+ H_2 (X_1)+ H_2(X_3)+1 \Big)\\
	&\quad -  \frac{1}{3} \Big(  H_2 (X_1) H_2(X_2)+ H_2 (X_1)+ H_2(X_2)+1 \Big) \in \KK[X_1,X_2,X_3],\\
	g_2(X_1,X_2,X_3)&= - \Gamma_Y (g_1 +p_1)= - \Gamma_Y (g_1 +1) = - \Gamma_Y (g_1) \\
	& = -\frac{1}{3} Y \Big( H_2(X_1)+H_2(X_2)+H_2(X_3) \Big)
 - 2 Y \in \KK[X_1,X_2,X_3],
\end{align*}
and
\begin{align*} 
	\Gamma_Y (g_2 (X_1,X_2,X_3)) &= -\Gamma_Y \left( \Gamma_Y (g_1) \right)\\
	& = -\frac{1}{15}\Big( H_2(X_1) X^2_2 X^2_3 + \frac{1}{3}  H_3(X_1) X_1  X^2_3 + \frac{1}{3} X_1 H_3(X_1) X^2_2 \Big)\\
	& \quad -\frac{1}{15}\Big( H_2(X_2) X^2_1 X^2_3 + \frac{1}{3}  H_3(X_2) X_2  X^2_3 + \frac{1}{3} X^2_1 X_2  H_3(X_3) \Big)\\
	& \quad -\frac{1}{15}\Big(  X^2_1 X^2_2 H_2(X_3) + \frac{1}{3} X^2_1  X_3  H_3(X_3) + \frac{1}{3} X^2_2 X_3  H_3(X_3) \Big)\\
	& \quad -\frac{1}{4} g_1(X_1,X_2,X_3) \in \KK[X_1,X_2,X_3].
	\end{align*}
	Hence, $g_3(X_1,X_2,X_3) +p_3(Y)= \Gamma_Y(g_2(X_1,X_2,X_3)) - 3 g_1(X_1,X_2,X_3) + Y^2 \neq 0$ almost surely, whilst $\E[g_3(X_1,X_2,X_3) + p_3(Y)\, \vert\, Y ] =0$ almost surely. In contrast to the product of two independent standard Gaussian random variables, we do not know yet whether the product of three or more independent standard Gaussian random variables admit any algebraic polynomial Stein operator. 

	\end{itemize}
	
	\end{ex}
	
\begin{rem} \label{rem:Stein-Chain-Space} 
	\begin{enumerate}
		\item[(i)] The following  algorithm can be used to produce the associated algebraic Stein chains:\\
		\hrule height 0.5pt depth 0.5pt width \textwidth
		\textbf{Algorithm} : Producing forward algebraic Stein chain
		\hrule height 0.5pt depth 0.5pt width \textwidth
		\begin{enumerate} 
			\item \textbf{At stage $t=0$}, pick a polynomial $p_0 \in \KK[Y]$ such that $\E[p_0(Y)]=0$, and set $g_0=-\E[ p_0(Y) ]=0$, 
			\item \textbf{At stage $t>1$}, with chosen polynomials $(p_0,p_1,\ldots,p_{t-1})$ in the subring $\KK[Y]$ satisfying moment property $\E[p_s(Y)]=-\E[g_s]$ for $s=0,\ldots,t-1$, if
			\begin{align*}
			g_t:= \Gamma_Y\bigl( g_{t-1} +  p_{t-1} (Y)  \bigr ) \in \KK[Y]
			\end{align*}
			set $ p_{t}(Y)=-g_t$ and stop.
			\item Otherwise, choose a polynomial $ p_{t}(Y)\in \KK[Y]$ with $\E[ p_{t}(Y) ]=-\E[g_t]$,
			and continue  to  stage $(t+1)$.
		\end{enumerate}
	\hrule height 0.5pt depth 0.5pt width \textwidth
		\item[(ii)] We have not yet explained how to choose the polynomials $p_t(Y)$
		and  how to determine whether algebraic Stein chains of finite length beginning with initial state $p_0(Y)$ do exist. This is in fact the topic of Section \ref{sec:Null}. The question of how to chose the initial state is also of interest. In the Malliavin-Stein method, it is often desirable to take $p_0(y)=cy$ as the zeroth-order term; however, other choices of initial state can lead to Stein operators with smaller degree $m$ or smaller order $T$  (see Table \ref{tablesummary}). Remark \ref{rem:MATLAB_code_General_0Order_Term} describes how initial states can be chosen to obtain such Stein operators.

\item[(iii)]	Instead of looking for polynomial coefficients in $\KK[Y]$,
we could restrict the coefficients to be in a linear subspace, as for example
\begin{align*}
\mathbb{L}=\KK_r[Y]=\bigl \{ p(Y) \, : \, \text{ polynomials with } \deg{(p)}\le r    \bigr \}.
 \end{align*}
 This observation is vital for our final goal of finding  an algebraic  Stein chain by implementing a \texttt{MATLAB} code (see Remark \ref{rem:Null-Issues}, item (ii)). We could also formulate the problem with a sequence of possibly different subspaces $\mathbb{L}_t \subseteq\KK[Y]$ at each stage $t\in \N$.   
  Note also that if
  \begin{align*}
  \big \langle \nabla h(X_1,\ldots,X_d) \big \rangle \cap \KK_r[Y]= \{ 0\},\end{align*}
  there cannot be any non-trivial  finite order algebraic polynomial Stein operator satisfying  $\deg p_t(y)\le r,$  $\forall t$. 
  \end{enumerate}
  \end{rem} 

  \subsection{Backward Stein chain}\label{sec3.2}
We may also validate an algebraic Stein chain by
 starting from the highest order term and checking recursively the lower order terms. We illustrate this subsidiary approach only in the univariate case $d=1$ (see Remark \ref{rem:BSC_d>1} for the multidimensional case). We start from  the highest order derivative coefficient
 $p_T( Y )\in \langle \partial h(X) \rangle \cap \KK[Y]$
(if this condition is not satisfied, we stop with a negative answer: there cannot be an algebraic polynomial Stein operator with highest order coefficient $p_T(Y)$). Let 
\begin{equation}\label{eq:division}
q_T(X)=-p_T( h(X))/\partial h(X).
\end{equation}
 After  $t$ successful recursion steps, we should have
\begin{align} \label{backward:condition}
 \delta q_{T-t}(X) 
 \in \left( \Big \langle \partial h(X) \Big \rangle
 + \KK[Y] \right).
\end{align}
If this condition is not satisfied we stop with a negative answer; otherwise, there is a polynomial of the target
$p_{T-t-1}(Y)$ such that $\bigl( \delta q_{T-t}(X)-p_{T-t-1}(h(X)) \bigr) \in \langle \partial  h(X) \rangle$.
We set
\begin{align*}
q_{T-t-1}(X)= \bigl( \delta q_{T-t}(X)-p_{T-t-1}(h(X)) \bigr) /\partial h(X)
\end{align*}
and continue.
The algebraic backward Stein chain ends successfully after  $(T-1)$ backward steps when
\begin{align*}  \delta  q_{T}(X) =
p_{0}( h(X) )=p_{0}(Y)\in \KK[Y].
\end{align*}
In such case
\begin{align*} &
 \E\bigl[ p_0( Y) f(Y) \bigr] = 
 \E\bigl[ \delta( q_1(X) )  f(Y) \bigr] =  
 \E\bigl[ q_1(X)  \partial h(X)  \partial f(Y) \bigr]  & \\ &
 = \E\biggl[  \frac{ \delta q_2(X)-p_1(Y) } {\partial h(X) }\partial h(X) \partial f(Y) \biggr]= -\E\bigl[ p_1(Y) \partial f(Y) \bigr] + \E\bigl[   \delta q_2(X)\partial f(Y) \bigr]
 = \dots  & \\ & = -\E\bigl[ p_1(Y) \partial f(Y) \bigr] - \E\bigl[ p_2(Y) \partial^2 f(Y) \bigr] - \dots
  -\E\bigl[ p_{T-1}(Y) \partial^{T-1} f(Y) \bigr]\\
  &\quad + \E\bigl[ \delta q_T(X) \partial^{T-1} f(Y) \bigr],
 & \end{align*}
with 
\begin{align*}
\E[ \delta q_T(X) \partial^{T-1} f(Y) \bigr]=\E[ q_T(X) \partial h(Y) \partial^{T} f(Y) \bigr]=
-\E[ p_T(Y) \partial^{T} f(Y) \bigr],
\end{align*}
which means that ${\mathcal S}f(y)= \sum_{t=0}^T  p_{t}(y) \partial^t f(y)$
is an algebraic polynomial Stein operator for the target $Y=h(X)$.

\begin{rem}\label{rem:BSC_d>1} 
	In this remark $X=(X_1,\ldots,X_d)$. The backward construction of an algebraic Stein chain in the multidimensional case $(d>1)$ is essentially the same as in the univariate case $d=1$ with only the minor difference that at each backward step the divisor polynomials $q_{T-t}, \, t=0,\ldots,T-1$, are not unique, unlike the univariate case. For example, start with the polynomial coefficient $p_T(Y)$ of the highest order term that must be
in $\KK[Y]\cap \langle Dh(X) \rangle$. (If this condition is not valid, we know there is no algebraic polynomial Stein operator with the highest order coefficient term $p_T$. This condition can be checked by
the multivariate  polynomial division algorithm, using Gr\"obner basis.) The latter means there is some polynomial vector 
$q_{T}(X)  = ( q_{T,1}(X),\dots, q_{T,d} (X)) \in \KK[X]^d$
such that
\begin{align*}
p_T(h(X) )= \sum_{k=1}^d q_{T,k}(X) D_k h(X).
\end{align*}
However, when the condition
is satisfied, (when $d>1$) the divisor polynomial $q_{T}(X)$ does not need to be unique. In fact, one can add any solution of the homogeneous equation
\begin{align*}
\sum_{k=1}^d u_{T,k}(X) D_k h(X) =0.
\end{align*}
This point has to be compared with relation \eqref{eq:division} in the univariate case $d=1$ where the divisor polynomial $q_T$ is unique. 
\end{rem}
\subsection{Null controllability }\label{sec:Null}
This section outlines the connection between Stein's method and the theory of linear systems. For a comprehensive account of the theory of linear systems,  the reader is referred to \cite{Rugh,Fuhrmann,fuhrmann2015,henri-1,henri-2} and references therein. Let  $\VV$ (state space) and $\Lbb$ (control space) be two vector spaces (not necessarily finite dimensional) over $\KK$. A linear discrete system $\Sigma$ is a quadruple $(\VV,\Lbb,\Gamma,\Lambda)$, where $\Gamma: \VV \to \VV$ (evolution operator) and $\Lambda:\Lbb \to \VV$ (input operator) are two linear maps, and the state variables follow the following linear dynamic 
\begin{align}\label{eq:Linear-System}
g_{t} = \Gamma g_{t-1} + \Lambda  p_{t-1} , \quad t =1,2,\ldots. 
\end{align}
Let us introduce the notion of (null) controllability in the theory of linear systems, another component that plays a significant role in our paper. 
\begin{mydef}\label{def:Null-Controllability}
Let $T\ge 1$. We say that an initial state $g_0\in \VV$ is null controllable in $T$-steps and denoted by 
$g_1 	\stackrqarrow{T}{\Sigma}  0$, 
if there is a finite (null) control sequence $\{  p_0, \dots, p_{T-1}  \}\subseteq \Lbb$  such that the linear recursion \eqref{eq:Linear-System} reaches $g_{T}=0\in \VV$ at time $T$. Denote 
\begin{align*}
\mathcal{N} (\Sigma,T) & = \left \{  g_0 \in \VV  \, : \, g_0 	\stackrqarrow{T}{\Sigma}  0 \right \}.
\end{align*}
The set of all null controllable states is $\mathscr{C}(\Sigma)  = \bigcup_{T\ge 1} \mathscr{C} (\Sigma,T)$. Clearly, both $\mathscr{C}(\Sigma,T)$ and $\mathscr{C}(\Sigma)$ are vector spaces over $\KK$, and $\mathscr{C}(\Sigma,T) \subseteq \mathscr{C}(\Sigma,T+1)$ for every $T \ge 1$.
\end{mydef}

\begin{lem}\label{lem:Null-Controllability-SimpleCrtierion}
Consider a linear discrete system $\Sigma$ as described above.  Then, the state $g_0 \in \VV$ is null controllable if and only if for some $T\ge 1$,
\begin{equation*}
\Gamma^T g_0  \in \sum_{t=1}^{T} \Gamma^{T-t} \Lambda (\Lbb).
\end{equation*}
	\end{lem}

\begin{proof}
	Using the linear dynamic \eqref{eq:Linear-System} one can readily  obtain that, for $t=1,\ldots,T$,
	\begin{align*} 
	g_t = \Gamma^{t} g_0 + \sum_{s=1}^{t}  \Gamma^{t-s} \Lambda p_{s-1}, 
	\end{align*}
	and hence, the result follows at once.
	\end{proof}
\begin{rem}\label{rem:Null-Controllability1}
	\begin{enumerate}
If the null controllability problem has finite horizon solutions  
the following algorithm  finds the shortest null control sequences:\\
\hrule height 0.5pt depth 0.5pt width \textwidth
\textbf{Algorithm} :  Finding null control sequence
	\hrule height 0.5pt depth 0.5pt width \textwidth
\begin{itemize}
\item[(a)] \textbf{At stage} $t=0$, pick an initial state $g_0 \in \VV$, and  consider the linear  equation
\begin{align*}
  \Gamma g_0 = -\Lambda p_0. 
\end{align*}
If this equation has a solution in $\Lbb$, then stop (equivalently, $\Gamma(-g_0) \in \text{Range}(\Lambda)$).
\item[(b)] Otherwise we continue the  recursion, until at some  stage $T$ for the first time the system
\begin{align}\label{linear:controllability:eq}
 \Gamma^T g_0 =-\bigl( \Lambda p_{ T-1} +  \Gamma  \Lambda  p_{  T-2} + \Gamma^2     p_{ T-3 } 
 + \dots  + \Gamma^{T-2} \Lambda   p_{ 1 } + \Gamma^{T-1}\Lambda p_0
\bigr)\end{align}
has solution (which is not necessarily unique) $(p_0,\dots, p_{T-1})\in \Lbb^T$.
\end{itemize}
	\hrule height 0.5pt depth 0.5pt width \textwidth
\end{enumerate}
\end{rem}

The linear control framework applies directly to our forward algebraic Stein chain construction. To see this, consider the linear system $\Sigma$, where the state space  $\VV= \KK[X_1,\dots, X_d]$  and the control space $\Lbb=\KK[Y]$. We stress that both the state and control spaces are infinite dimensional vector spaces which do not enjoy suitable analytic properties such as being a Banach space. Moreover, the evolution and input operators are given by the Malliavin operator ($X=(X_1,\ldots,X_d)$),
\begin{align*}
\Gamma_Y : \KK[ X_1,\dots, X_d]\to  \Big \langle \partial_{X_1} h(X),\ldots, \partial_{X_d}h (X) \Big \rangle 
\subseteq \KK[ X_1,\dots, X_d].
\end{align*}
The following result is an immediate consequence of the definitions of an algebraic Stein chain and a null controllable state.

\begin{prop}\label{prop:Chain=Null}
Fix $T\ge 1$. Let $d \ge 1$, and assume that the Gaussian polynomial target variable $Y$ takes the form \eqref{eq:Polynomial-Target}. Let $\Sigma=(\VV,\Lbb,\Gamma_Y,\Gamma_Y)$ be the linear system as described above. Let 
\begin{equation}\label{eq:initial-space}
\KK_0[Y]= \Big \{  p(Y) \in \KK[Y] \, : \,  \E[p(Y)]=0 \Big \}.
\end{equation}
\begin{itemize}
\item[(a)] Then, there is a surjective binary relation (and not a function) between the set $\mathscr{C}(\Sigma,T) \cap \KK_0[Y]$ and the set $ASC(Y,T)$ of all algebraic Stein chains of length $T$. Moreover, there is a surjective binary relation between the set $\mathscr{C}(\Sigma)\cap \KK_0[Y]$ and the set $APSO(Y)$ of all the algebraic polynomial Stein operators for the target random variable $Y$. 
\item[(b)] When $d=1$, the same statement as in item (a) holds by replacing everywhere $ASC(Y,T)$ with $SC(Y,T)$, and $APSO(Y)$ with $PSO(Y)$, respectively.
\end{itemize}
\end{prop}

\begin{proof}
(a) For a given $g_0  \in \mathscr{C}(\Sigma,T) \cap \KK_0[Y]$, and $p=(p_t(y) \in \KK[y] \, : \, t = 0,\ldots,T-1) \in ASC(Y,T)$, we say that $g_0 R p$ if the sequence $p$ is a null control sequence for $g_0$. 
Clearly, $R$ defines a  surjective (but not injective) relation that is not a function, since  in general null control sequences (even of the same length) are not unique.  The other statement follows directly from Propositions \ref{prop:Chain-Leads-SteinOperator}, and \ref{prop:Algebraic_Stein_Operators}. (b) Apply Proposition \ref{prop:Algebraic_Stein_Operators}, item (b). 
\end{proof}

\begin{rem}
Assume the convention $X=(X_1,\ldots,X_d)$ as above. Let $p_0 \in \Lbb=\KK[Y]$. Then, the target random variable $Y=h(X)$ admits an algebraic polynomial Stein operator with zeroth order term $p_0$ if the  initial state   $g_0=g_0(X):=\Gamma_Y p_0( h(X) )$ can be null controllable in a finite time. Mathematically, for some $t\ge 1$, there exist polynomials  $p_0(Y),\dots, p_{t-1}( Y ) \in \KK[Y]$ such that  
		\begin{align}\label{eq:g_t=IteratedGammaOperatros}
		 g_{t+1}(X)= \Gamma^t g_0(X) +  \sum_{s=1}^{t}   \Gamma_Y^{t-s} p_{s-1}(h(X) ) =0.
		\end{align}
The latter is  a direct consequence of the set inclusion $\Gamma^t_Y (\KK[Y])  \subseteq \sum_{s=1}^{t} \Gamma^{t-s}_Y (\KK[Y])$ which is in order as soon as the ascending chain $I_1 \subseteq \cdots \subseteq I_t \subseteq I_{t+1} \subseteq \cdots $ would stop, where $I_t = \sum_{s=1}^{t} \Gamma^{t-s}_Y (\KK[Y])$ for $t \in \N$. Although, the Hilbert base theorem guarantees that $\KK[X]$ is a Noetherin ring, however we cannot conclude that the chain would necessary stop due to the undesirable fact that the  vector spaces $I_t$ are not ideal.	
	\end{rem}
\begin{rem}\label{rem:Null-Issues}
	\begin{itemize}
\item[(i)] In order to apply the rich theory of linear control systems in our framework, we mention the following issues. Firstly, in our formulation, both the state and the control spaces are infinite dimensional vector spaces over $\KK$. Secondly, it is also possible that one can think of the (infinite dimensional) Hilbert state space $\VV=L^2(\mu)$ with $\mu$ the standard Gaussian measure on the real line. However, the control space $\Lbb$ ``must" be taken as the ring $\KK[Y]$, which is not a Hilbert space (not even a Banach space). Therefore, we believe that the theory developed for infinite dimensional linear systems, see for instance \cite{Fuhrmann,Trig1,Trig2}, is hardly applicable in our setting to study the null controllability.
\item[(ii)] 

To bypass the obstacles mentioned above,  in the  computer  implementation of the algorithm we may  set an upper bound on the polynomial degree of the Stein operator coefficients,  assuming that the initial state and controls belong to $\Lbb=\KK_m[Y]$. Then, at stage  $t$,  
\begin{align*} \deg( g_t ) \le \left\{ \begin{matrix} 
\deg(h)\times m + (\deg(h) -2) \times t
& \mbox{(univariate $h$)}\\
\deg(h)\times (m+t)  & 
\mbox{ (multivariate $h$).} \end{matrix}\right. 
\end{align*}
Hence, in our linear system, state space is time varying (increasing in time); however, within a finite time horizon $T$ we can implement the algorithm to check null controllability up to time $T$ on a finite dimensional state space $\R^{\deg(h)\times m +  ( \deg(h)  - 2) \times T}$ and control space $\R^{\deg(h)\times m }$, see also Remark \ref{rem:Stein-Chain-Space}. Moreover, we point out that even by fixing a time horizon $T$, the Kalman criterion as described in the introduction)
 cannot be used to study null controllability, because the criterion checks null controllability within the whole null states $\mathscr{C}(\Sigma)$ and not on $\mathscr{C}(\Sigma,T)$. Lastly,   in the computer implementation of the linear system $\Sigma$ (as described above) the input operator $\Lambda = \Gamma_Y \Theta $, where the operator $\Theta$ is given by the embedding 
$p(Y)\in \KK[Y]$ into $p( h(X_1,\dots, X_d))\in\KK[X_1,\dots,X_d]$. This is due to the fact that operator $\delta^{-1}$ (in the definition of the $\Gamma_Y$ operator) acts over the polynomial ring $\KK[X_1,\ldots,X_d]$. 
\end{itemize}  
\end{rem}

\begin{rem}\label{rem:MATLAB_code_General_0Order_Term}
	In the Malliavin-Stein method, the zeroth-order polynomial $p_0(y)=cy$ is commonly used. However, other choices of $p_0(y)$ may allow for Stein operators with either lower order $T$ or lower maximal coefficient degree $m$.	 In the computer implementation of our algorithm, for a given target $Y=h(X)$
		the input is a zeroth-order polynomial coefficient $p_0(h(x))$, with 
		$\E[p_0( h(X) ) ]=0$. In order to find a Stein operator with generic $p_0(y)$ 
		we run the algorithm several times with the respective  initial coefficients
		$p_0^{(k)}(y)= y^k - \E[ h(X)^k]$, in order to obtain the corresponding
		Stein operators satisfying
		\begin{align*}
		\E\bigl[  {\mathcal S}^{(k)} f( Y) \bigr]=-\E\bigl[ p_0(Y) f(Y)\bigr]
		\end{align*}
		for $k=1,\dots, m_0$. When the Stein chains end successfully
		and Stein operators are found for the initial coefficients above,
		by solving a linear system for
		the linear combination one can further reduce the order
		or the maximal coefficient degree  of a Stein operator
		with a generic zeroth-order polynomial coefficient $p_0(y)$ of degree $\le m_0$.
	\end{rem}

\section{ Applications to Gaussian Hermite polynomials }\label{sec4}

\subsection{Highest order polynomial coefficient}\label{sec:Higher-Coef}

Before the next proposition, we note the following lemma that will be needed in the proof.  The result of \cite{ind61} is stated for the physicists' Hermite polynomials rather than the probabilists' Hermite polynomials, as used in our paper.

\begin{lem}\label{indlem}(\cite{ind61}). Define $E_p(x)=(\pi^{1/2}2^{2p} p!)^{-1/2}e^{-x^2/4}H_p(x)$.  Then the relative maxima of $|E_p(x)|$, $x\geq0$, steadily increase, i.e., if $x_1<x_2<\ldots<x_j$ are the non-negative zeros of $E_p'$ for fixed $p$, then
	\begin{equation*}|E_p(x_1)|<|E_p(x_2)|<\ldots<|E_p(x_j)|.
	\end{equation*}
\end{lem}

\begin{prop} \label{prop:higher-coeff} Let $X \sim N(0,1)$. Assume that $Y=H_p(X)$, where $H_p$ is the Hermite polynomial of degree $p \ge 2$ (the case $p=1$ corresponds to standard Gaussian distribution that is not of interest in this paper). Let 
$	\mathcal{S} = \sum_{t=0}^{T} p_{t} \partial^t \in PSO(Y) =APSO(Y)$ be a polynomial Stein operator for $Y$. As before, identify the (dependent) variable $y=H_p(x)$. Then the following properties hold:
\begin{enumerate}
\item[(a)] $\KK[ y ]\cap \langle H'_p(x) \rangle$ is an ideal of the subring $\KK[y]$. Moreover,
	\begin{align*}
	\KK[ y ]\cap \langle H'_p(x) \rangle = \langle t(y) \rangle
	\end{align*}
	with 
	\begin{align*}
	t(y)= \prod_{z: H'_p(z)=0}  \bigl( y- H_p(z) \bigr).
	\end{align*}
		\item[(b)]	$p_{T}(y) \in \KK[ y ]\cap \langle H'_p(x) \rangle$, and $p_{T}( H_p(x) )=0$  for all 
	solutions $x$ of  $H'_p(x)=0$. In particular, $p_{T}(y)=0$ at all local minimum or local maximum values $y$ of $H_p$. 
	\item[(c)]  $\mathrm{deg}(p_{T})\geq p/2$ if $p$ is even, and $\mathrm{deg}(p_{T})\geq p-1$ if $p$ is odd.
	\end{enumerate}
\end{prop}
\begin{proof} 
	(a) It is easy to see that $\KK[ y ]\cap \langle H'_p(x) \rangle \subseteq \KK[y]$ is an ideal of the subring $\KK[y]$. Also, after substitution $y=H_p(x)$, it becomes clear that the polynomial $t(y)=t(H_p(x))$, given by
	$
	t( y ) = \prod_{z: H'_p(z)=0} \bigl( y -H_p(z) \bigr) \in \KK[y],  
	$
	is divisible by $H'_p(x)$ by using the Taylor expansion $H_p (x) - H_p (z) = \sum_{k=1}^{p} H^{(k)}_p (z) (x-z)^k$ for every $z$ such that $H'_p (z)=0$. Note that $H'_p(x)=p H_{p-1}(x)$ and all the roots of the  Hermite polynomials are real. So the claim follows at once by a direct application of Remark \ref{rem:PID}. 

\vspace{2mm} 

\noindent{(b)} Apply Corollary  \ref{cor:Leading-Coff_Contain_h'(X)}. The rest are direct consequences. 
		
\vspace{2mm} 		
		
\noindent{(c)}  This follows from  part (b) and the fact that $H_p$ has $p/2$ distinct values for the local maxima and minima when $p$ is even, and $p-1$ distinct values for the local maxima and minima when $p$ is odd.  It is a standard property of $H_p$ that it has exactly $p$ real roots.  Therefore $H_p$ must have $p$  local maxima and minima when $p$ is even and $p-1$ local maxima and minima when $p$ is odd.  But when $p$ is even, $H_p$ is an even function and there can hence be at most $p/2$ distinct values for the local maxima and minima.  That there are exactly $p/2$ distinct local maxima and minima in the even case now follows from the stronger result of Lemma \ref{indlem}.  For odd $p$, we let $x_1<x_2<\ldots<x_{(p-1)/2}$ be the non-negative zeros of $H_p'$, and as $H_p$ is an odd function we have that $-x_{(p-1)/2}<\ldots<-x_{2}<-x_1$ are the negative zeros of $H_p'$.  Then due to the stronger result of Lemma \ref{indlem}, we have that $|H_p(x_1)|=|H_p(-x_1)|<\ldots<|H_p(x_{(p-1)/2})|=|H_p(-x_{(p-1)/2})|$ with $H_p(x_k)=-H_p(-x_k)$ for all $k=1,2,\ldots,(p-1)/2$.  Thus, there are exactly $p-1$ distinct values for the local maxima and minima in the odd case. 
\end{proof}

\begin{rem} It is well-known that (see, \cite[Proposition 2.1]{gaunt-34}) target random variables $Y=h (X)$ with $h \in \KK[X]$ and $\text{deg}(h) =2$ admit a polynomial Stein operator of order two with linear polynomial coefficients.
	
\end{rem}

\begin{rem} Part (c) of Proposition (\ref{prop:higher-coeff}) is useful in implementing our code to find Stein operators for $H_p(X)$.  In particular, if $p$ is even we must seek polynomial Stein operators with coefficients that have degree at least $p/2$, and if $p$ is odd we require the degree to be at least $p-1$. As can been seen in Table \ref{tablesummary}, we have used our \texttt{MATLAB} code to find polynomial Stein operators for $H_p$, $p=1,\ldots,10$, that attain these minimum possible degrees.  In fact, we have also tested this for $p=11,12,14,16,18$, and our \texttt{MATLAB} code has always been able to find a Stein operator with the minimum possible degree.  It therefore seems reasonable to conjecture that this is the case for all $p\geq1$.
\end{rem}

\subsection{Examples of Stein operators for $H_p(X)$}\label{sec4.4}

For univariate Gaussian polynomials, our \texttt{MATLAB} code can find all polynomial Stein operators up to a given order $T$ and maximum degree $m$.  We illustrate this in Appendix \ref{appendixa} by providing examples of Stein operators for $H_p(X)$, $p=1,\ldots,6$, where $X\sim N(0,1)$.  The code can also be applied to $h(X)$ when $h$ is not a Hermite polynomial, as demonstrated in the Supplementary Information (SI).  

In this paper, we only give the `simplest' Stein operators for $H_p(X)$, $p=1,\ldots,6$.  For $p=1,\ldots,6$, we list the Stein operators
with the lowest order $T$ and lowest degree $m$ (if there are two Stein operators with the lowest $T$ ($m$), we given the one with lowest $m$ ($T$)).  For $p\geq7$, the Stein operators become too complex to state in this paper, but in Table \ref{tablesummary} we give a summary of the complexity of the `simplest' Stein operators for $p=1,\ldots,10$.   Many other examples are given in the SI, some of which are `simpler' in other senses than those presented in Appendix \ref{appendixa}, such as having lower values of $T+m$ or $T\times m$.

\begin{table}
	\caption{Summary of ordered pairs $(T,m)$ of the order $T$ of the Stein operator and degree $m$ of polynomial coefficients for the Stein operators for $H_p(X)$ that either minimise $T$ or $m$.  The CPU and elapsed time (in seconds)  correspond to the pair $(T,m)$ of zero-order term $cyf(y)$ with Min $m$.
	}\label{tablesummary} 
	\centering
	\begin{tabular}{|l|l|l|l|l|l|l|l|l|l|l|}
		\hline
		\multirow{2}{*}{Distribution} &
		\multicolumn{2}{c|}{General 0th-order term} &
		\multicolumn{4}{c| }{0th-order term $cyf(y)$} \\
		& Min $T$ & Min $m$ & Min $T$ & Min $m$  & CPU  & Elapsed time\\
		\hline
		$H_1(X)$ & (1,1) & (1,1) & (1,1) & (1,1) & -  &  -\\
		$H_2(X)$ & (1,1) & (1,1) & (1,1) & (1,1) & 0.4& $0.5\,\si{\second}$ \\
		$H_3(X)$ & (3,4) & (5,2) & (4,3) & (5,2)& 0.7 & $0.9\,\si{\second}$  \\
		$H_4(X)$ & (2,3) & (3,2) & (3,2) & (3,2) & 0.6 & $0.8\,\si{\second}$ \\
		$H_5(X)$ & (5,12) & (13,4) & (6,11) & (13,4)  & 1.4& $4.6\,\si{\second}$ \\
		$H_6(X)$ & (3,6) & (6,3) & (4,5) & (6,3)  &1.3 &$2.1\,\si{\second}$ \\
		$H_7 (X)$  & (7,24) &   (25,6)  &   (8,23)   &  (25,6)  & 5.8& $44.4\,\si{\second}$  \\
		$H_8 (X)$  & (4,10)  &   (10,4)   &   (5,9)   &  (10,4) & 1.8 & $5.2\,\si{\second}$ \\
		$H_9 (X)$  & (9,40) &   (41,8)  &   (10,39)   &      (41,8) & 21.0 & $343.3\,\si{\second}$\\
		$H_{10} (X)$  & (5,15) &  (15,5)   &  (6,14)    &   (15,5) & 3.6 & $21.8\,\si{\second}$ \\
		\hline
	\end{tabular}
\end{table}

The Stein operators in Table \ref{tablesummary}, except the one for $H_9(X)$ with $(T,m)=(9,40)$, were obtained using a standard laptop, a \textit{MacBook Pro with processor: 2,9 GHz Dual-Core Intel Core i5	and memory:  8 GB, 2133 MHz, LPDDR3}.  Despite the complexity of the operators, the code found them rather quickly. For example, it was able to find the horribly complex Stein operator for $H_9(X)$ with $(T,m)=(41,8)$ in just $343.3\,\si{\second}$. We easily obtained the Stein operator for $H_9(X)$ with $(T,m)=(9,40)$ using a powerful computer server at The University of Manchester, and, whilst not explored in this study, a full exploitation of such computational power could yield some hugely complex Stein operators!  Indeed, using just a standard PC, we obtained a Stein operator for $H_{20}(X)$ with zero-order term $cyf(y)$ and $(T,m)=(11,54)$, with 402.6 CPU and elapsed time $4276\,\si{\second}$.

There are several interesting observations we can make from the Stein operators presented in Appendix \ref{appendixa} and the summary in Table \ref{tablesummary}. We notice  that there is an important increase in complexity from $p=1,2$ to $p=3,4$, in which the Stein operators go from being first order with linear coefficients (for which it is simple to solve the corresponding Stein equation) to be being at least second order with higher order coefficients (for which we are not able to solve the corresponding Stein equation).  There is a further increase in complexity from $p=3,4$ to $p=5,6$, in which the Stein operators go from being expressed in one line equations to equations that sprawl several lines.  


From Table \ref{tablesummary}, we observe that, for 0-order term $cyf(y)$ the ordered pairs $(T,m)$ satisfy the following recipes:

\begin{itemize}

\item [(a)] Minimum $T$: (i) when $p\ge 4$ is even, we have $(T,m)= (p/2 +1, {p/2 \choose 2} + p/2 -1)$; (ii) when $p\geq3$ is odd, we have $(T,m)=(p+1,\binom{p}{2}+(p-1)/2-1)$.

\item [(b)] Minimum $m$: (i) when $p\geq2$ is even, we have $(T,m)=(\binom{p/2+1}{2},p/2)$; (ii) when $p\geq3$ is odd, we have $(\binom{p+1}{2}-(p-1)/2,p-1)$.

\end{itemize}

For general $0$-order term we have:

\begin{itemize}

\item [(c)] Minimum $T$: (i) when $p$ is even, just pick up the pair $(T,m)$ minimizing $m$ with $0$-order term $cy f(y)$ associated with that value of $p$  and switch the components; (ii) when $p$ is odd, pick up the pair $(T,m)$ minimizing $T$ with $0$-order term $cy f(y)$ associated with that value of $p$  and set $(T-1,m+1)$.

\item [(d)] Minimum $m$: for both even and odd $p$ this is the same as in item (b).

\end{itemize}

Rather curiously, the minimum values of $T$ seem to be connected with the number of distinct local maxima and minima of the Hermite polynomial $H_p$.  We know that this is the case for the minimum possible degree $m$, due to Proposition \ref{prop:higher-coeff}, part (c).  

We observe that for each $p$ we tested the code always found a Stein operator with the minimum possible degree (see Proposition \ref{prop:higher-coeff}, part (c)).  We do not, however, have a proof of an analogous result for the minimum possible order $T$ of Stein operators.  In theory, it is therefore possible that we have not actually found Stein operators with the minimum possible values of $T$.  However, our tests suggest that this is a remote possibility.  For example, for zero-order term $cyf(y)$, our code could not find any Stein operators for $H_4(X)$ with input variables $T=2$ and $m=60$, nor for $H_5(X)$ with input variables $T=5$ and $m=80$. Additionally, we performed many other tests with our code and verified all of the predicted recipes (a) -- (d) for a number of values of $p$ between 11 and 20 (the complex Stein operator for $H_{12}(X)$ is given in the SI).  As such, we believe that it is reasonable to conjecture that the recipes (a) -- (d) hold for all $p$.

\appendix

\section{Stein operators for univariate Gaussian Hermite polynomials}\label{appendixa}

The Stein operator for $H_1(X)$ is the classical standard Gaussian Stein operator of \cite{stein}.  The Stein operator for $H_2(X)$ is a special case of the Stein operator for $aX^2+bX+c$, $a,b,c\in\mathbb{R}$, of \cite{gaunt-34}.
Also, the Stein operators (\ref{h3x}) and (\ref{h4x}) were already obtained by \cite{gaunt-34}.  All other Stein operators in this appendix 
are new.  Many more examples are given in the Supplementary Information.


\vspace{3mm}

\noindent{$H_1(X)$ and $H_2(X)$ :}	
 \begin{align*} 
y-\partial, \quad y - (2\,y + 2)\partial
 \end{align*}

\noindent{$H_3(X)$ :}
\begin{align}\label{h3x}
   y -6\partial
   -99\,y\partial^{2}
   +(216-27\,y^2)\partial^{3}
   +486\,y\partial^{4}
   +(486\,y^2-1944)\partial^{5} 
\end{align}
 \begin{align*}  
 ( 290\,y-y^3 )
+ ( 528\,y^2-1560)\partial
+ ( 243\,y^3-1404\,y )\partial^{2}
+ ( 27\,y^4-648\,y^2+2160 )\partial^{3}
  \end{align*}

\noindent{$H_4(X)$ : }
\begin{align} \label{h4x}
y -(24+44\,y) \partial  +(576+144\,y-16\,y^2)\partial^{2} +(192\,y^2+576\,y-3456)\partial^{3}
\end{align}
  \begin{align*}
  ( -y^2+50\,y+24 )
 + (64\,y^2+72\,y-1008 )\partial
 + (16\,y^3-48\,y^2-576\,y+1728 )\partial^{2}
  \end{align*}

\noindent{$H_5(X)$ :}
\begin{align}&y- 120\partial- 75325\,y\partial^2+ (- 81875\,y^2+7704000 )\partial^3+ (- 31250\,y^3+270600000\,y)\partial^4\nonumber\\
&+ (- 3125\,y^4 + 527800000\,y^2 - 39086400000)\partial^5+ (280000000\,y^3 - 155065000000\,y)\partial^6 \nonumber\\
&+ (35000000\,y^4 - 241335000000\,y^2 + 14306880000000)\partial^7\nonumber\\
&+ (- 198750000000\,y^3+53403600000000\,y )\partial^8\nonumber\\
&+ (- 33125000000\,y^4 + 34950000000000\,y^2 - 1170432000000000)\partial^9\nonumber\\
&+ (39000000000000\,y^3 - 10843200000000000\,y)\partial^{10}\nonumber\\
& + (9750000000000\,y^4 - 6696000000000000\,y^2 + 352512000000000000)\partial^{11}\nonumber\\
& + (- 2160000000000000\,y^3+622080000000000000\,y )\partial^{12}\nonumber\\
&+(- 1080000000000000\,y^4 + 622080000000000000\,y^2 - 29859840000000000000)\partial^{13}\nonumber
\end{align}
 \begin{align*}
& (y^9-104800744\,y^7+174104044032\,y^5-82431615212544\,y^3+9617056740900864\,y) 
\\& +(-83053520\,y^8+191761742080\,y^6-148596701936640\,y^4
+33440484399022080\,y^2\\
&-868706901405204480)\partial
 \\& +(-23029125\,y^9+72332912000\,y^7-88767223008000\,y^5
 +32039796049920000\,y^3\\
 &-1984593650909184000\,y)\partial^{2}
  \\& +(-2831875\,y^{10}+11857320000\,y^8-22211556000000\,y^6
  +11983543971840000\,y^4\\
  &-1826589574103040000\,y^2+54875902433034240000)\partial^{3}
  \\& +(-156250\,y^{11}+855800000\,y^9-2353387200000\,y^7
  +1868056934400000\,y^5\\
  &-530407371571200000\,y^3+36302379968102400000\,y)\partial^{4}
  \\ &+(-3125\,y^{12}+22000000\,y^{10}-85519200000\,y^8
+99156326400000\,y^6\\
&-65065321267200000\,y^4+19243712957644800000\,y^2-849260402284953600000)\partial^{5}  
\end{align*}

\noindent{$H_6(X)$ :}
\begin{align} &
y + (- 1278\,y - 720)\partial + (- 972\,y^2 + 103320\,y + 756000)\partial^2 \nonumber \\ &
+ (- 216\,y^3 + 228960\,y^2 + 16491600\,y - 120528000)\partial^3 \nonumber \\ &
+   (71280\,y^3 + 6771600\,y^2 - 307152000\,y - 3265920000)\partial^4  \nonumber \\ &
+   (- 314928000\,y^2 - 19945440000\,y + 125971200000)\partial^5 \nonumber \\ 
&+   (- 209952000\,y^3 - 19945440000\,y^2 + 251942400000\,y + 7558272000000)\partial^6 \nonumber    
   \end{align} 
  \begin{align*} &
 ( 15303970800 y - 252586320 \, y^2 -6227803 \, y^3 + 599 \, y^4 ) 
  \\ &
 +  ( -6722792640000 -28723248000 \, y   +   30858084000 \, y^2   - 247410960 \, y^3  -6390132 \, y^4) \partial 
 & \\ & 
 + (  -25152249600000     -8314215840000 \, y  +  29111400000 \, y^2 +  14157844200   \, y^3       
  \\  &
 -43020180\, y^4 -1746684\, y^5  ) \partial^2 
 & \\ & 
+ (  1173771648000000   + 27946944000000 \, y     -3912572160000 \, y^2  -13197168000 \, y^3  
\\   &
+ 1633473000\, y^4     -129384 \, y^6 ) \partial^3
 \end{align*} 

\section{Auxiliary lemma}\label{appendixc}
	
	\begin{lem} \label{lemma:conditional:expectation}
	Fix $d=1$.  Let $X\sim N (0,1)$. Assume that $g(X),h(X)\in \mathbb{K}[X]$ with $\mathbb{K}=\mathbb{Q}, \mathbb{Z}$.  
	Then with $Y=h(X)$,
	\begin{align}\label{assumption:root:lemma}
	  \mathbb{E}\big[ g(X) \,  \big \vert \,  h(X) = y\big] =0, \quad \forall y \in \mathrm{supp}(Y)\;  
	\end{align}
if and only if one of the two
conditions below are satisfied:
\begin{enumerate}
\item[(a)]  $g\equiv 0$, i.e., $g$ is identically the zero polynomial.
\item[(b)]  $g$ is an odd polynomial and  $h$ is an even polynomial, meaning that  $g(-x)= - g(x)$ and $h(x)=h(-x)$ for all $x\in\mathbb{R}$.
\end{enumerate}
\end{lem}
\begin{proof} 
 Note that \eqref{assumption:root:lemma}  is equivalent to
\begin{align}
\label{assumption:root:lemma:1}
 \sum_{x\in\mathbb{R}: h(x)=y } g(x)e^{-x^2/2 } = 0,   \quad \forall y\in \mathrm{supp}(Y).
\end{align} 
Next we discuss on the degree of polynomial $h$.  Suppose first that $h$ has odd degree, in which case $h$ is not an even polynomial.
Then for all $y$ with $\vert y \vert $ large enough the algebraic equation $h(x)=y$ has one and only one real root, let us say $r(y)$, and moreover the mapping  $y\mapsto r(y)$ is strictly monotone  outside
a compact interval. Then relation \eqref{assumption:root:lemma:1} yields that the polynomial $g$ has infinitely many real roots, which implies that  $g\equiv 0$.

Now, consider the case with $h$ of even degree. Then, for all $\vert y\vert $ large enough,  depending on the signs of $y$ and  of the leading coefficient of $h$ either equation $h(x)=y$ has no real roots or it
has only two real roots, let us say $r_1(y), r_2(y)$ so that  $r_1(y) < 0 < r_2(y)$. Moreover, as before the mappings  $y\mapsto r_1(y), y\mapsto r_2(y)$ are strictly monotone outside a compact interval. Next assume that $g \not\equiv 0$. In the latter case relation \eqref{assumption:root:lemma:1} means that for all large enough $y$:
 \begin{align}\label{eq:root:lemma:1}
    \frac{g(r_1(y) )}{ g(r_2(y) ) } = -\exp\biggl( \frac{r_2(y)^2-r_1(y)^2}{2} \biggr).
 \end{align}
Now, by choosing $y\in \mathbb{Q}$, the roots $r_j(y)$, $j=1,2$, of the algebraic equations $h(x)=y$, are algebraic numbers (because $\mathbb{K}=\mathbb{Z},\mathbb{Q}$), and hence the left-hand side of relation \eqref{eq:root:lemma:1} is an algebraic number since the quotient of two algebraic numbers is again an algebraic number. This implies that the right hand side of relation  \eqref{eq:root:lemma:1} is an algebraic number and that in virtue of  the Lindemann-Weierstrass theorem \cite[Theorem 1.4]{baker} this only happens when $r_2 (y) = \pm r_1 (y)$.  But our sign justification  yields that must be $r_1(y)=-r_2(y)$ for all large enough $y\in\mathbb{Q}$.  Therefore, from the relation  \eqref{eq:root:lemma:1} one can infer that $g(r_2(y)) = -  g (r_1 (y))$ for all large enough $y\in\mathbb{Q}$. Now consider polynomial $G(x): =  g (x)  + g (-x)$. Then, this polynomial would have infinitely many roots, and hence $G(x)=0$ for all $x \in \mathbb{R}$, i.e., $g$ is an odd polynomial.  Looking to a new polynomial $H(x) := h(x) - y$, and with a similar argument one can conclude that $h(x) = h(-x)$ for all $x \in \mathbb{R}$, i.e., $h$ is an even polynomial. 
\end{proof}

\section*{Acknowledgements}
We would like to thank Ivan Nourdin for first bringing to our attention the fascinating problem of finding Stein operators for $H_n(X)$, $n\geq3$.  Without these initial conversations, this paper would not exist. EA would also like to thank Peter Eichelsbacher and  Yacine Barhoumi-Andr\'eani for many stimulating discussions on Stein's method. We also would like to
thank our enthusiastic readers Kaie Kubjas and Luca Sodomaco for their comments and remarks. We thank anonymous referees for their constructive comments that improved the quality of this paper.
RG was supported by a Dame Kathleen Ollerenshaw Research Fellowship.

\newpage

\huge

\appendix

\title{ \bf Supplementary Information}

\normalsize

\section{MATLAB implementation}\label{sec3.4}

This appendix is devoted to a concise description of the implementation of the null controllability approach in \texttt{MATLAB} to find polynomial Stein operators for Gaussian polynomial distributions of the form $Y=h(X)$ in the univariate case $(d=1)$, for which our code returns many new polynomial Stein operators. Some of them are listed in Appendices \ref{appendixa} and \ref{apring}. Our exposition in this appendix mostly focuses on polynomial Stein operators of the form 
\[\mathcal{S} = \sum_{t=0}^{T} p_t \partial^t\]
having  zeroth-order term $p_0 (y) = cy$ for some constant $c$. These are often the most useful Stein operators, in particular, in the context of the Nourdin-Peccati Malliavin-Stein approach.  In Remark \ref{rem:MATLAB_code_General_0Order_Term}, we describe how the implementation of our code is modified for generic zeroth-order coefficients $p_0(y)$. 

\vspace{3mm}

The body of the \texttt{MATLAB} code consists of the following:
\begin{description}
	\item[(a)] {\bf Input variables}: the code requires the following three input variables,
	\begin{itemize}
		\item Polynomial $h=h(x) \in \KK[x]$; this input variable corresponds to the target distribution $Y=h(X)$, a Gaussian polynomial random variable with $X \sim N (0,1)$. We further assume that the polynomial $h$ is chosen in a such way that $\E[Y]=0$. Prototypical examples are $h=H_p$, the Hermite polynomials with $p \ge 1$.
		\item Variable $T$ stands for the maximum time horizon to check the null controllability, and corresponds to order of the polynomial Stein operator. The program starts checking the null controllability with $t=1$ onwards until $t=T$, and stops at step $t$ as soon as it finds a null control sequence. Otherwise, the code returns with an error in the case there is no a null control sequence of the maximum  length $T$. In that case, one has to enlarge the time horizon.
		\item Variable $m$ is the maximum degree of the null control sequence $(p_t (y) \, : \, t=0,\ldots,T)$ in the target variable $y$, i.e. $m = \max_{t=0,\ldots,T} \deg (p_t)$. In the language of algebraic geometry, the input variable $m$ yields the control state to be $\KK_m[y]$ (Remark 3.7, item (ii)).
		Again, the code returns with an error in the case there is no a polynomial null control sequence with maximum degree $m$. In that case, again, one has to enlarge $m$.
	\end{itemize}

	\begin{lstlisting}[
	style      = Matlab-editor,
	basicstyle = \mlttfamily,
	]{stein_chain.m}
	%%% input variables %%%
	h=h(x);     
	%%%  corresponds to the target Gaussian polynomial  Y= h(N) with E[Y]=0, for example h=hermite(4,x) %%%
	T=10;           %%% time horizon %%%
	m=4;            %%% maximum degree of  polynomial coefficients of Stein operator  in the target varaible y=h(x) %%%
	\end{lstlisting}

	\item[(b)] {\bf Output variables} :
	
	\begin{itemize}
		\item  Vector $\texttt{uu= [c,u]}$ stores the coefficients of the polynomials $(p_t : t=0,\ldots,T)$ of the found Stein operators. This output variable is useful when degree of the input polynomial $h$ is large, translating to Hermite polynomials $h=H_p$, when  $p$ is large. This enable us to observe semi-directly the polynomial coefficients of the found Stein operator, and avoid the printing since it is quite time consuming.
		\item Variable  \texttt{t} in the workspace assigns the order of the found Stein operator.  
		\item Variable \texttt{m1} in the workspace assigns the maximum degree of the polynomial coefficients associated to the found Stein operator. 
		\item The end-part of the code is devoted to the printing of the found Stein operator in \LaTeX as well as the symbolic format $\sum_{t=0}^{T} p_t (y) D^t$. This is done through the output variable \texttt{stein_eq}. 
	\end{itemize} 
	\begin{ex}
		(a) With input variables $h(x)=H_4(x)= x^4 -6x +3$, $T=5$, and $m=2$, the \texttt{MATLAB} code in the workplace returns the above listed output variables as follows: \[ \texttt{uu} = [ \underbrace{1}_{\text{coefficients of } p_0}, 	\underbrace{-24,	-44,	0, 0}_{\text{coefficients of } p_1}\, ,   	\underbrace{576,	144,	-16, 0, 0}_{\text{coefficients of } p_2}\, ,	\underbrace{-3456,	576, 192, 0, 0}_{\text{coefficients of } p_3}]. \] 
		This corresponds to polynomial coefficients: $p_0(y)=y$, $p_1(y)= -24 -44 y$, $p_2(y)=576 + 144y-16y^2$ and $p_3(y)=-3456+576y+192y^2$, as well as $\texttt{t}=3$, $\texttt{m1}=2$, and 
		\[  \texttt{(192*y^2+576*y-3456)*D^3+(-16*y^2+144*y+576)*D^2+(-44*y-24)*D+y}. \] 
		This is the Stein operator (\ref{h4x}), which is listed in  Appendix \ref{app:cyf(y)}.
		
		\vspace{3mm}		
		
		\noindent{(b)} With input variables $h(x)=H_3(x)= x^3 -3x$, $T=4$, and $m=4$, the \texttt{MATLAB} code in the workplace returns the above listed output variables as follows:
		\begin{align*} &\texttt{uu} =\\
			& [\underbrace{5}_{\text{coefficients of } p_0},	\underbrace{-12,		0,		-3,		0,		0}_{\text{coefficients of } p_1} \, , \underbrace{	0,	207,		0,		0}_{\text{coefficients of } p_2}\, , \underbrace{			-1080,		0,		351,		0,		0}_{ \text{coefficients of } p_3}\, , 		\underbrace{	0,	-324,		0,		81,		0, 0}_{\text{coefficients of } p_2}]
			. 
		\end{align*}
		This corresponds to polynomial coefficients: $p_0(y)=5y$, $p_1(y)= -12 -3 y^2$, $p_2(y)= 270y$, $p_3(y)=-1080+351y^2$ and $p_4(y)=-324y+81y^3$, as well as $\texttt{t}=4$, $\texttt{m1}=3$ and 
		\[ \texttt{(81*y^3-324*y)*D^4+(351*y^2-1080)*D^3+207*y*D^2+(-3*y^2-12)*D+5*y}. \] 
		This is the new Stein operator (\ref{h3x-new}), which is listed in  Appendix \ref{app:cyf(y)}.
	\end{ex}
	
	\item[(c)] {\bf Pseudocode for the null controllability problem} :
	
	The following Pseudocode constitutes the main body of the code. It consists of two parts. The first part implements in a recursive method the crucial relation \eqref{linear:controllability:eq} to find a null control sequence starting from $t=0$ until $t=T$. The second part checks that the moment conditions \eqref{eq:moments-check} are satisfied. As above, input variable $T$ is the upper bound on the order of the Stein operator, $\Lambda$ is the control matrix and $\Gamma$ is the
	evolution matrix. \texttt{Nmoments} is the vector $\E[ X^k]_{k\ge 1}$, $v_0$ is the coefficient  vector of the given zeroth-order polynomial term $p_0(h(x))$
	in the independent variable $x$ satisfying $\E[ p_0( h(X) )]=0$. For example, when $h=H_4 (x)=x^4-6x^2+3$, and $p_0(y)=cy$, then $v_0 = c(3,0,-6,0,1)$. 
	
	\begin{lstlisting}[
	style      = Matlab-editor, mathescape,
	basicstyle = \mlttfamily, 
	caption={[]Pseudocode for the null controllability problem:}
	]{stein_chain.m}
	$t=0$ 
	$v=v0$
	$A=\Lambda$
	go_on=true
	while go_on
	$t=t+1$
	$v=\Gamma v$
	$u=$linsolve$(A,-v)$  %%% solving  Au=-v
	if $(t<T)$&$(Au=-v$ has no solutions$)$
	$A=[\Gamma*A \quad \Lambda']$   %%% matrix concatenation
	go_on=true
	else
	go_on=false
	end    
	end
	if ($Au=-v$ has solution after $t\leq T$ steps)   
	%%%   imposing moment conditions on the polynomial coefficients
	$v=v0$
	index=1:(m+1)
	for(tau=1:(t$-$1))
	$v=\Lambda u($index$)+\Gamma v$
	$u$(index(1))=($u$(index(1))$-$Nmoments$(1:Q)v)$ 
	index=index+(m+1)
	end
	end
	\end{lstlisting}
	
	\begin{rem}\label{rem:MATLAB_code_General_0Order_Term}
		The algorithm above for a given target $Y=h(X)$
		takes as input the zeroth-order polynomial coefficient $p_0(h(x))$, with 
		$\E[p_0( h(X) ) ]=0$. In order to find a Stein operator with generic $p_0(y)$ 
		we run the algorithm several times with the respective  initial coefficients
		$p_0^{(k)}(y)= y^k - \E[ h(X)^k]$, in order to obtain the corresponding
		Stein operators satisfying
		\begin{align*}
			\E\bigl[  {\mathcal S}^{(k)} f( Y) \bigr]=-\E\bigl[ p_0(Y) f(Y)\bigr]
		\end{align*}
		for $k=1,\dots, m_0$. When the Stein chains end successfully
		and Stein operators are found for the initial coefficients above,
		by solving a linear system for
		the linear combination one can further reduce the order
		or the maximal coefficient degree  of a Stein operator
		with a generic zeroth-order polynomial coefficient $p_0(y)$ of degree $\le m_0$.
	\end{rem}
\end{description}

\section{Stein operators for univariate Gaussian Hermite polynomials}\label{appendixa}

The Stein operator for $H_1(X)$ is the classical standard Gaussian Stein operator of \cite{stein}.  The Stein operator for $H_2(X)$ is a special case of the Stein operator for $aX^2+bX+c$, $a,b,c\in\R$, of \cite{gaunt-34}, as is the Stein operator for $H_1(X)+H_2(X)$ given in Appendix \ref{apring}.  Also, the Stein operators (\ref{h3x}) and (\ref{h4x}) were already obtained by \cite{gaunt-34}.  All other Stein operators in this appendix and Appendix \ref{apring} are new. Some of the Stein operators given in this appendix are also given in Appendix A. We repeat them here so that all the Stein operators are collected together in a single appendix.  

\subsection{Zeroth-order term $cyf(y)$}\label{app:cyf(y)}

\noindent{$H_1(X)$ and $H_2(X)$ :}	
\begin{align*} 
	y-\partial, \quad y - (2\,y + 2)\partial
\end{align*}

\noindent{$H_3(X)$ :}
\begin{align} \label{h3x-new}
	5\, y -(3\,y^2+12)\partial   +207\,y\partial^{2}
	+(351\,y^2-1080)\partial^{3}
	+(81\,y^3-324\,y)\partial^{4}
\end{align}
\begin{align}\label{h3x}
	y -6\partial
	-99\,y\partial^{2}
	+(216-27\,y^2)\partial^{3}
	+486\,y\partial^{4}
	+(486\,y^2-1944)\partial^{5} 
\end{align}

\noindent{$H_4(X)$ : }
\begin{align} \label{h4x}
	y -(24+44\,y) \partial  +(576+144\,y-16\,y^2)\partial^{2} +(192\,y^2+576\,y-3456)\partial^{3}
\end{align}

\noindent{$H_5(X)$ :}
\begin{align*} 
	&15580403168538081808023552 \,y \\
	&+(-319179359200955\,y^8+17296743383000046809080\,y^6\\
	&-30453944634963174391774080\,y^4+8378454869686262588172134400\,y^2\\
	&-171397515591740804005791498240)\partial
	& \\ & +(16154152521786318001600\,y^7-41918287476242535868569600\,y^5\\
	&+18160711517167770618284313600\,y^3-1239101680729174664564404224000\,y)\partial^{2}
	& \\ & +(5177408385598691055000\,y^8-19166255615207862008920000\,y^6\\
	&+13403727175244880138330240000\,y^4
	-2442403318372552645917235200000\,y^2\\
	&+65725928416658664921713541120000)\partial^{3}
	& \\ & +(724337658286667253125\,y^9-3713592558790019185200000\,y^7\\
	&+4280839659909236338428000000\,y^5
	-1379071236058382967127603200000\,y^3\\
	&+66328376785356945138012979200000\,y)\partial^{4}
	& \\ & +(44884597387634296875\,y^{10}-310405490799058194000000\,y^8\\
	&+580995055104909396324000000\,y^6-283708061282453615759001600000\,y^4\\
	&+33287571883579035278551449600000\,y^2-906872445506697193065598156800000)\partial^{5}
	& \\ & +(997435497502984375\,y^{11}-9089418036906323000000\,y^9\\
	&+26589273781119330420000000\,y^7-19596186937165037285376000000\,y^5\\
	&+4565266162550649874870272000000\,y^3-190391529130477012565360640000000\,y)\partial^{6}
\end{align*}
\begin{align*} 
	& 3387150029792y   + (571033901y^4 - 547887964579112y^2 + 27078999070435200)\partial   
	\\ & + (378923502108829200y - 1502636158422420y^3)\partial^2 
	& \\ & +  (- 1206880149763500y^4 + 576041545375798000y^2 - 17414865657146035200)\partial^3
	& \\ & + (- 367959969956875y^5 + 443885745185225000y^3 - 77181050945393280000y)\partial^4
	& \\ & + (- 44612023515625y^6 + 158102370728325000y^4 - 72273712209585600000y^2\\
	& + 3073114984002508800000)\partial^5
	& \\ & +  (- 1784480940625y^7 + 20659537391900000y^5 - 19467313520040000000y^3\\
	&+ 2959552681754572800000y)\partial^6
	& \\ & +
	(732950666875000y^6 - 4688334677088000000y^4 + 2477569953587328000000y^2 \\
	&- 117950656010379264000000)\partial^7   
\end{align*} 
\begin{align}&y- 120\partial- 75325\,y\partial^2+ (- 81875\,y^2+7704000 )\partial^3+ (- 31250\,y^3+270600000\,y)\partial^4\nonumber\\
	&+ (- 3125\,y^4 + 527800000\,y^2 - 39086400000)\partial^5+ (280000000\,y^3 - 155065000000\,y)\partial^6 \nonumber\\
	&+ (35000000\,y^4 - 241335000000\,y^2 + 14306880000000)\partial^7\nonumber\\
	&+ (- 198750000000\,y^3+53403600000000\,y )\partial^8\nonumber\\
	&+ (- 33125000000\,y^4 + 34950000000000\,y^2 - 1170432000000000)\partial^9\nonumber\\
	&+ (39000000000000\,y^3 - 10843200000000000\,y)\partial^{10}\nonumber\\
	& + (9750000000000\,y^4 - 6696000000000000\,y^2 + 352512000000000000)\partial^{11}\nonumber\\
	& + (- 2160000000000000\,y^3+622080000000000000\,y )\partial^{12}\nonumber\\
	&+(- 1080000000000000\,y^4 + 622080000000000000\,y^2 - 29859840000000000000)\partial^{13} \nonumber
\end{align}

\noindent{$H_6(X)$ :}
\begin{align*} & 599 \,y +(-218\,y^3+913612\,y^2+53550492\,y-281527920)\partial   \\ & +(1336776\,y^3+104875908\,y^2-1387746360\,y-28764115200)\partial^{2} 
	& \\ & +(494424\,y^4+41703336\,y^3-1035418680\,y^2-29104855200\,y+158972544000)\partial^{3}
	& \\ & +(47088\,y^5+3490128\,y^4-321541920\,y^3-16820071200\,y^2+241351488000\,y\\
	&+6178654080000)\partial^{4}
\end{align*} 
\begin{align*} &
	631y + (- 1632y^2 + 1469778y + 720720) \partial   + (4106412y^2 + 245648160y - 1739944800)\partial^2   \\ & +  (2507544y^3 + 227216880y^2 - 4235230800y - 79890624000)\partial^3 & \\ & + (352512y^4 + 37318320y^3 - 2919790800y^2 - 198457128000y + 1006369920000)\partial^4 & \\ & 
	+(- 1907064000y^3 - 181171080000y^2 + 2288476800000y + 68654304000000)\partial^5       
\end{align*} 
\begin{align} &
	y + (- 1278\,y - 720)\partial + (- 972\,y^2 + 103320\,y + 756000)\partial^2 \nonumber \\ &
	+ (- 216\,y^3 + 228960\,y^2 + 16491600\,y - 120528000)\partial^3 \nonumber \\ &
	+   (71280\,y^3 + 6771600\,y^2 - 307152000\,y - 3265920000)\partial^4  \nonumber \\ &
	+   (- 314928000\,y^2 - 19945440000\,y + 125971200000)\partial^5 \nonumber \\ 
	&+   (- 209952000\,y^3 - 19945440000\,y^2 + 251942400000\,y + 7558272000000)\partial^6  \nonumber   
\end{align}

\noindent{$H_7(X)$ :}
\begin{align*}
	&y -5040\partial
	-133451451\,y\partial^{2}
	+(1064591216640-306955845\,y^2)\partial^{3}
	\\& +(849390311141064\,y-306156312\,y^3)\partial^{4}
	\\& +(-116825457\,y^4+2701312236798840\,y^2-8436683256242088960)\partial^{5}
	\\ & +(-17294403\,y^5+3327495514534800\,y^3-852383157281913430656\,y)\partial^{6}
	\\ & +(-823543\,y^6+1495553659757640\,y^4-2882447388687770112384\,y^2\\
	&+7910371125237853207756800)\partial^{7}
	\\ & +(251232195145176\,y^5-4800644752916327503104\,y^3\\
	&+505980701265753819516063744\,y)\partial^{8}
	\\ & +(13222747112904\,y^6-2727340858843315440768\,y^4
	+837637190655055545978961920\,y^2\\
	&-1777134045147718167454078402560)\partial^{9}
	\\ & +(-536928170359162643712\,y^5+1281752157950575496201318400\,y^3\\
	&-110523786891733518638811362611200\,y)\partial^{10}
	\\ & +(-31584010021127214336\,y^6+913589345116795680281088000\,y^4\\
	&-176461821057357744953672419061760\,y^2+340345632611067543386816669260185600)\partial^{11}
	\\ & +(214893775695296588083845120\,y^5
	-191362567154678131174434027601920\,y^3\\
	&+14300407319501324117284327929271910400\,y)\partial^{12}
	\\ & +(14326251713019772538923008\,y^6
	-170427454945767764472321169551360\,y^4\\
	&+32709717771471627211087752733986816000\,y^2\\
	&-65419147430764591344326824530412044288000)\partial^{13}
	\\ & +(-49486748509729074183341479464960\,y^5
	+16693019644892749062546112354050048000\,y^3\\
	&-439054884730049001199187142951434649600000\,y)\partial^{14}
	\\ & +(-3806672962286851860257036881920\,y^6
	+7954423881903475224494820925678387200\,y^4\\
	&-1284347964264936234207319478284579307520000\,y^2\\
	&+2375893125471619127467751045131020533760000000)\partial^{15}
	\\ & +(2626852987014389750244682334756044800\,y^5\\
	&-690264828368716626786026864105558507520000\,y^3\\
	&+3957829778041610249387356594675097665536000000\,y)\partial^{16}
	\\ & +(238804817001308159113152939523276800\,y^6\\
	&-177049564531965021489200207455680921600000\,y^4\\
	&+18066493172730934812816336030401279950848000000\,y^2\\
	&-26665162994119386969582865827453121131970560000000)\partial^{17}
	\\ & +(-58746597218584833047955018294421094400000\,y^5\\
	&+14203481064016267590268922250061037174784000000\,y^3\\
	&-6268263770190783711518376340458529515110400000000\,y)\partial^{18}
	\\ & +(-6527399690953870338661668699380121600000\,y^6\\
	&+2686174908845226527905185370947095494656000000\,y^4\\
	&-104373689539033328254646376972599062953984000000000\,y^2\\
	&+49997032312081377455344473295518420864860160000000000)\partial^{19}
\end{align*}
\begin{align*}
	\\ & +(728248810874643441914034877572271177728000000\,y^5\\
	&-170689150758390923069718861514581646560460800000000\,y^3\\
	&+368356231367805828906903541263381229448724480000000000\,y)\partial^{20}
	\\ & +(104035544410663348844862125367467311104000000\,y^6\\
	&-32822200050800981502352944003864349586227200000000\,y^4\\
	&+325525490017310356319310959603263156833484800000000000\,y^2\\
	&+204605928378464524860045079684944299315272089600000000000)\partial^{21}
	\\ & +(-4349052313418574507170272863754793779200000000000\,y^5\\
	&+981649950497815495375611300357280589225656320000000000\,y^3\\
	&-4914738800909755586442003506056569993818013696000000000000\,y)\partial^{22}
	\\ & +(-869810462683714901434054572750958755840000000000\,y^6\\
	&+256185589276368632570648020348415300736122880000000000\,y^4\\
	&-2344464557378731854989940379852673228570886144000000000000\,y^2\\
	&+4185814371312008995677288668201630441603176857600000000000000)\partial^{23}
	\\ & +(7646575073392101937433120802707182599536640000000000\,y^5\\
	&-1486494194267424616636998684046276297349922816000000000000\,y^3\\
	&+7135172132483638159857593683422126227279629516800000000000000\,y)\partial^{24}
	\\ & +(2548858357797367312477706934235727533178880000000000\,y^6\\
	&-743247097133712308318499342023138148674961408000000000000\,y^4\\
	&+7135172132483638159857593683422126227279629516800000000000000\,y^2\\
	&-10274647870776438950194934904127861767282666504192000000000000000)\partial^{25} 
\end{align*}

\noindent{$H_8(X)$ :}
\begin{align*} 
	&y +(-56496\,y-40320)\partial
	+(-65920\,y^2+282800000\,y+2564997120)\partial^{2}
	\\ & +(-32768\,y^3+646180864\,y^2+721579048960\,y-20217560186880)\partial^{3}
	\\ & +(-4096\,y^4+422596608\,y^3+32899399680\,y^2-2200644332027904\,y\\
	&+4130252292096000)\partial^{4}
	\\ & +(61931520\,y^4-523683299328\,y^3-2574503360593920\,y^2-366232179868434432\,y \\ & +58545839107190292480)\partial^{5}
	\\ & +(-114809110528\,y^4-1287467746983936\,y^3-1268709319526842368\,y^2
	\\ & +1817998267095743201280\,y+22309748508410354073600)\partial^{6}
	\\ & +(-196254369841152\,y^4+569368730639794176\,y^3+2830846067667621642240\,y^2
	\\ & -604629551552531949158400\,y-54650034667549385293824000)\partial^{7}
	\\ & +(273194783751536640\,y^4+715893205099641569280\,y^3-380939584507781264179200\,y^2
	\\ & -24325934978807793451008000\,y+6733738025236644715560960000)\partial^{8}
	\\ & +(-4239094388568883200\,y^4-68279093316679001702400\,y^3\\
	&-109137943708451267936256000\,y^2
	+32788493016294424093655040000\,y\\
	&+1187169574727901562011648000000)\partial^{9}
	\\ & +(-28486714291182895104000\,y^4-75974067014584781242368000\,y^3
	\\
	& +32411613786222074391429120000\,y^2
	+2532628426086189998958182400000\,y\\
	&-531851969478099899781218304000000)\partial^{10}
\end{align*} 

\begin{align*}
	& 43881259694506575260585152325108256000\, y\\& +(758341924998600395427477\,y^6-3036680132069054017863854184592\,y^5  \\ & +1331741571102189252273068158184708\,y^4\\
	&+9939276590962003336259277645051422576\,y^3
	\\ & -26008516350439618070893556661572208948480\,y^2\\
	&-986182705922231246369585722972814390123520\,y
	\\ & +323088290686531584414330815058804702874828800)\partial
	\\ & +(-2767717490299691868790081096992\,y^6
	+2745681813845247197112891663934592\,y^5\\
	&+8688054800046337238130243859447301504\,y^4\\
	&
	-52179633604735703361563733004233081400320\,y^3\\
	&-4006558543353497248668551812037023095674880\,y^2\\&+4985984568788688257608326512776182631260979200\,y\\
	&+141211093900944841687067636916910780575989760000)\partial^{2}
	\\ & +(-795470345646531870787606273920\,y^7
	+1388929044391096197930808274579200\,y^6\\
	&+1268106261890252292545440864336329216\,y^5\\
	&
	-28615644819544623493285720342091267558400\,y^4\\
	&-2825848698037716449581319232123919549747200\,y^3\\
	&+9060918738384649710556682129785161795003187200\,y^2\\
	&-715292771515588119123981713302937352253440000000\,y\\&-106207629904083491235044661177045119718823034880000)\partial^{3}
	\\ & +(-86972718694239482150786482176\,y^8
	+231078212024682780911571646642176\,y^7\\
	&-152076843749162059500082566159024128\,y^6\\
	&-5387079504425678805357610745518014996480\,y^5\\
	&+27077797312871915258743890674262374645760\,y^4\\
	&+5163356044327843006697287768165841717113651200\,y^3\\
	&-989489331574519523427436336482810207234686976000\,y^2\\
	&-227274837232294053495518331417863538508415631360000\,y\\
	&+15412063011799677536642694195548287434134598451200000)\partial^{4}
	\\ & +(-3106168524794267219670945792\,y^9
	+11495907371738602251043051245568\,y^8\\
	&-22387748849180188556746006330867712\,y^7\\
	&-307924407132723321325224896118597599232\,y^6\\
	&+95040341941716935401536503884726932602880\,y^5\\
	&+746279388930738753144950205257832651050188800\,y^4\\
	&+15726861419239583669372242939835501912260608000\,y^3\\
	&-126295034113768734113066635718873264523092951040000\,y^2\\
	&-4386479742103932925514515787520191827399409664000000\,y\\
	&+1743269240872982562749436619330409204491480989696000000)\partial^{5} 
\end{align*}

\newpage 

\noindent{$H_{10}(X)$ :}

\begin{align*}
	&y +(-3257160\,y-3628800)\partial
	+(-5304000\,y^2+1091578185900\,y+17240138496000)\partial^{2}
	\\ & +(-4455000\,y^3+2869770870000\,y^2+60285174403756000\,y-8195889073034880000)\partial^{3}
	\\ & +(-1250000\,y^4+3038747400000\,y^3-87788841797700000\,y^2
	-24178492814093711990000\,y\\
	&+386738413261030656000000)\partial^{4}
	\\ & +(-100000\,y^5+983618100000\,y^4-307962932577300000\,y^3
	-39718345618893322400000\,y^2\\
	&+541197273193916916857500000\,y+77038387315300617650291900000)\partial^{5}
	\\ & +(86562000000\,y^5-137113471467000000\,y^4
	-20241095010910164000000\,y^3\\
	&+614674588478073576820000000\,y^2
	+38655573271292047153412480000000\,y\\
	&-1544077167741439083270654318000000)\partial^{6}
	\\ & +(-14034919410000000\,y^5-3611923606425530000000\,y^4\\
	&+1157055078389606630470000000\,y^3
	+115536421986640510338352080000000\,y^2\\
	&-1172509218379091417042657993760000000\,y-203744329919466653004638691354210000000)\partial^{7}
	\\ & +(-205210677603600000000\,y^5+637855102579767139100000000\,y^4\\
	&+56868672620562577951574600000000\,y^3
	-2272132852237024646250474124200000000\,y^2\\
	&+3128079515715829689098507996214600000000\,y\\
	&+3930024106272461026784153353683692400000000)\partial^{8}
	\\ & +(77901423775037568000000000\,y^5+658712542202571286381000000000\,y^4
	\\
	&-1292597318190806918874425331000000000\,y^3\\
	&-11812794761700707114492464433490000000000\,y^2\\
	&+129201682923315582458928485163725052000000000\,y\\&+18906431439212235333040203434783311488000000000)\partial^{9}
	\\ & +(-1142484431509952294400000000000\,y^5
	-305749394642557911039155820000000000\,y^4\\
	&-15381578151743373056375084568780000000000\,y^3\\&+295149927028628970469825239541761540000000000\,y^2\\
	&-605437932378359772612075849323517691200000000000\,y\\
	&-390628545918898871170275525535644040435200000000000)\partial^{10}
	\\ & +(-26218962977299345612800000000000000\,y^5\\
	&-1545311793922203552750743995800000000000\,y^4\\&+207378661186088279917311114396243000000000000\,y^3\\
	&-947337616600928539545447183238335921600000000000\,y^2\\
	&-6041201844736807513102134017326568627942400000000000\,y\\
	&+1616831588573884649913842569493503523008512000000000000)\partial^{11}
	\\ & +(283690569292007580831744000000000000000\,y^5\\
	&+42665527458982366228219924459476000000000000\,y^4\\
	&+619219707587936613335940342714893844000000000000\,y^3\\
	&-9516516580955534393084024006695235048832000000000000\,y^2\\
	&-13910396790204097895146662874181073365729280000000000000\,y\\
	&+12538450183646982969209287097523831183622963200000000000000)\partial^{12}
\end{align*}
\begin{align*}
	& +(2013925290187650123497472000000000000000000\,y^5\\
	&+239840927597396622568797477572567040000000000000\,y^4\\
	&-2256928407014913684359947598020212003840000000000000\,y^3\\
	&-9159749296064537337968038115498856606228480000000000000\,y^2\\
	&-325527100323490529261822120694127103095603200000000000000\,y\\
	&+11991506150131887654155034891290034767280734208000000000000000)\partial^{13}
	\\ & +(369569978445186098528256000000000000000000000\,y^5\\
	&+8850830771867151377434984714503936000000000000000\,y^4\\
	&-3743436852225958906407862862763267430656000000000000000\,y^3\\
	&+23226302082946533231478400179955443276627968000000000000000\,y^2\\
	&+53172555147902396289433672851669087631666053120000000000000000\,y
	\\&-15019660145846781121396718242753782824837106892800000000000000000)\partial^{14}
	\\ & +(-13969745185228034524368076800000000000000000000000\,y^5\\
	&-1654702318704542057314786132806225715200000000000000000\,y^4\\
	&+16406416015498811819438783476653145708339200000000000000000\,y^3\\
	&+52289565038882113895457269799258362330048102400000000000000000\,y^2\\
	&-3337702285103141267631792041784157267928442470400000000000000
	0000\,y\\
	&-63082572879429801056239769926313724421483203133440000000000000000000)\partial^{15} 
\end{align*}

\newpage

\begin{landscape}
	\noindent{$H_{12}(X)$ : }
	\begin{align*}
		&y +(-233070696\,y-479001600)\partial
		+(-501527520\,y^2+6352383744527616\,y+204648095606169600)\partial^{2}
		\\
		&
		+(-636840576\,y^3+19089382891822848\,y^2-3394883655397923928704\,y-7249139343152084444774400)\partial^{3} \\
		& +(-302351616\,y^4+29373160502264832\,y^3-85308646011131425582080\,y^2
		-602523817566740264004389743104\,y 
		\\ & +41064019096384578833796856564224)\partial^{4} \\
		& +(-53747712\,y^5+15840179948482560\,y^4-273427650869271786381312\,y^3-1316669788754566472307081885696\,y^2
		\\ & 
		+2170010164038440090108658752637868032\,y
		+309548697528486679478857232565800288256)\partial^{5}
		\\
		& +(-2985984\,y^6+3072326411759616\,y^5-192603344706608023830528\,y^4
		-560892309661592769908335767552\,y^3\\
		&
		+5116274212579508069139116444393668608\,y^2
		+1871508167720062294658849638594476994584576\,y\\
		& -1403246570516801969736707373238745847370149888)\partial^{6} 
		\\ & +(181886772215808\,y^6-42774112637105452056576\,y^5+
		132310250049749704550105776128\,y^4 \\
		& +6100892163566890030473183513504940032\,y^3
		+12185247995219038527151138220401114723123200\,y^2\\
		&-11681666228455173412459205357641945421177005473792\,y
		-2738295128641680419377595592608618596091386484981760)\partial^{7}
		\\ & +(-2748567267530379362304\,y^6+73666060474705985574221709312\,y^5+
		3451367123887058235380347068033662976\,y^4\\
		&+8195649080503140381789439036119782904692736\,y^3
		-42362348958827837192910542745637350140681572319232\,y^2\\
		&+8605805288052435217028645819998584260567085268807581696\,y+10090283275069314967490216320513438307130001993020258582528)\partial^{8} 
		\\ & +(6122969931529093833262891008\,y^6
		+765012229406473525731341611490082816\,y^5
		-2064970287577535546470829735560945847500800\,y^4\\
		&-40318034061293611210193101392041269667476460273664\,y^3
		+14715081803838418518675662117709963424788629914516979712\,y^2\\
		&+2675727389969723961227986471432673065534160692920819135283200\,y\\
		&-3405627775616047625932892137492620230683891635631185243869282304)\partial^{9}
		\\ & +(51751156056169776707893351620280320\,y^6
		-1469929441066575219994945831108173990199296\,y^5 \\&
		-12305087661174167627145989282915838084324110893056\,y^4+584
		2752166077012696722570334721588384028128144171991040\,y^3\\
		&
		+19959822612297584879093863766096386724511501378396587553718272\,y^2\\
		&-370974392074285692304211850136362507889245222926072505762761618
		2272\,y\\
		&-4262425425397040175599184428595121659767009470758388212160869892620288)\partial^{10}+ 
	\end{align*}
	
	\newpage
	
	\begin{align*} &(-139400148828847595375505273202823575633920\,y^6
		-1198110617665708229758579896899171499600456450048\,y^5\\&+3864843583243655623642649348584814778132564227446013952\,y^4
		+29204850945035863326516923657504643054893483414171947131469824\,y^3\\
		&-113230323517446671785407572316029270682042485518515524916729694
		12608\,y^2\\
		&+1312515419909214478025009333667826087304095063026659585732499391433
		93280\,y\\&+2153295952028741761283932908501607728677383634406121832787232023668846493696)\partial^{11}
		\\ & +(-23027463138850296119921836630997141664150061056\,y^6
		+1614540854760445386005973579051662961965790397137420288\,y^5\\
		&+10463994762437125780538670376093783335682135426771438968766464\,y^4\\&
		-8882286225885204709167635447927866792790882511764827220532007
		010304\,y^3\\
		&+1712462530078396982114539268435601301587597519130895155260522
		62895943680\,y^2\\
		&+32376271460679950157900528718805939630902324454206936085025687849
		8062157217792\,y\\
		&-73215262354744155216967956968845863715936225293876888907420793742682231268179968)\partial^{12}
		\\ & +(160519272911161893123315179349672359639594586788069376\,y^6
		+783957391000124418497400292797678152479143735398930451529728\,y^5\\
		&-2772356211648110431251286417307460864210898986567870555037158604
		800\,y^4\\
		&-22945330278343845440337595383088422562493168911580065733484
		43539822870528\,y^3\\
		&+10250554382635013363407008124976494382276664029360165260540444279
		78849218723840\,y^2\\
		&+40522759272397326650169812717287573278444776496207223792806837
		310380338077262413824\,y\\
		&-191323189577922031730414137928237076640693102295870007414031362545528658305782325968896)\partial^{13}
		\\ & +(-27843739476171697981780618685997614529206051307892840071168\,y^6\\
		&
		-43355029420852847512776214827693948065850463777170402003709329
		4080\,y^5\\
		&-1496672105919003175172152636369691134976836393737518735756406
		063605219328\,y^4\\
		&+5310933511639829524944852011998831152496181783109333280651531312
		03865673203712\,y^3\\
		&+66743964868774942220342749708982392339293713546929027542390760261
		477808901527175168\,y^2\\
		&+6867422932433728433986616897045299370088701065499148786784501505
		81298232291039509479424\,y\\
		&-11093384290324424891421388905615016062973277027421190297319718096764100234101735410234294272)\partial^{14}+
	\end{align*}
	
	\newpage
	
	\begin{align*}
		&(-2765073307298780331399057993417432427148798263235550296605
		6157184\,y^6\\
		&-1919367884109622613276723980544725054632732520546896733304499205
		00432896\,y^5\\
		&+2887201894635508180875000700429793534318841127681124392237482506
		3848328822784\,y^4\\
		&+5346080298330596887748734815696748570640247686930225145082936661
		8505698151887798272\,y^3\\
		&-6697382662729972182457374300779905786689691499506763324654395546
		714520583466089726345216\,y^2\\
		&-3426166289961184499929433858505345643148641543375091805933864536
		18226612647447115537089298432\,y\\
		&+869155706116682864138166059914832891289446863852185510770162845771311474897967612200903707197440)\partial^{15}
		\\ & +(-110229217660344648616227106732263572551547790631073947586686462879
		3344\,y^6\\
		&-34876340624435720982941871171485265231643122210627409753477513
		89343109873664\,y^5\\
		&+2678698074407307079541849749459713446374398391378425946833108296
		4896450338501951488\,y^4\\
		&-76954955996251396804468484966526093264758064403618275499117587
		32134226596100901798150144\,y^3\\
		&-8647988212173340465735441626976460724892326703802521661716088326
		31856134996651950394698629120\,y^2\\
		&-12324993322220043866764529513812685632683146926335758661567
		109807209912773281647759538980397252608\,y\\
		&+124417910760357151754014609393935038525550097370003037558479065559687666061731691858473897820420046848)\partial^{16}
		\\ & +(543549805276029974005398210561388341071511119572048980556219565
		230321041408\,y^6\\ &+377687005858630104409773297584561340389223295149500502091053
		3112680546297893093376\,y^5\\
		&-4651677650736016933837888253739872474358796316535598787266334
		16008003881868460602425344\,y^4\\
		&-5584048780307894426354474904873267941096052783536749933959
		06153714121751490402297891860250624\,y^3\\
		&-14701566935175991390312550425692671615673215853895624820466
		376095675933277884379186580478675124224\,y^2\\
		&+95730205921761572179042421445005142930457479230691954504984876
		8715492785776402400376440186396278259712\,y\\
		&+3580159500284274965213999250295506032937122478119078776591573901120017148989130374684574477064383949176832)\partial^{17} +
	\end{align*}
	
	\newpage
	
	\begin{align*}
		&+(2955399927136249736458047794626398611631282505673132276789334278
		7975517248159744\,y^6
		\\ &+184170677749548860070158635001220474781414959479804518512684047
		795936412924689296392192\,y^5\\
		&-151895563682802617426162056449726952653341148950151313554990525
		746259223841761950678764224512\,y^4\\
		&+2021282657899078884120428501271262937564689867027069664634888
		8251362383748998240980149483696291840\,y^3\\
		&+1694315493667460877258628702630330877147908897015076350601436
		920563245036999509801907228422264680611840\,y^2\\
		&-26328496236459868825825453762314707600363119492039390909109358206
		948160187666205409353070233189252471980032\,y\\
		&-234326050588925531828325584560803025498194298859111264310613909945027997365344471499000501804565124330591092736)\partial^{18}
		\\ & +(-2296060986440683341448312345342019280770684151507088633570051653
		719564677183596134400\,y^6
		\\ &-15182023375231210596276356784819080872835836768967037986067
		975217661858201950536238891008000\,y^5\\
		&+659745257353190882486422943625944972729731543998179808519707877416
		7003926663805592023172173004800\,y^4\\
		&+390018584390417072307167823345301543691260525095250713486827938892
		390184320942841640082617246741954560\,y^3\\
		&-15282099828058008864702325755388577836636092711097075416559932
		927829298887401689741170274267164683955863552\,y^2\\
		&+53171586708462210440095514526350361191393206116689681350478491473
		272569500821058360564883360042050600032010240\,y\\
		&+2416068861368139736567419743882470538184223139119307069078128318873270365448886390704687168156549021016022790963200)\partial^{19}
		\\ & +(5091171455575978697058183841590138719540137248954807733028647
		747435800940617180119040000\,y^6
		\\ & +3879610339885870244176387905315471207404170558534569709464806833622
		6973216247051165945036800000\,y^5\\
		&+13677767199212386471830875216964525447845530309844379738570506441
		630230299702482441761298086625280000\,y^4\\
		&-10507044205032703672291524045707812424116216789381900593110932827
		7912474876331868741581
		02892208148237516800\,y^3\\
		&-445170119930278273256646684976965512923450100272529446552100
		981289166754811030490048066655970718723373373849600\,y^2\\
		&+1029158534811991599632738690606269599410572173471594132213633192
		7554854373781157937331968420811218322449613455360000\,y\\&+28626820677537827909199986922139439469978383153456543812669322899551784568077573722817815829644430560173641842032640000)\partial^{20}+
	\end{align*}
	
	\newpage
	
	\begin{align*}
		&+(16935030208678269486457041051647491104009418248542136285729615
		62441024126106665293027737600000\,y^6
		\\ &+112128090387941806105254185221213033031509501069867286744601682676
		25841565617712011985854188748800000\,y^5\\
		&-4777177301325184118650055381719444446607562682108139143497789
		095144992793317958765814245292159847628800000\,y^4\\
		&-31700423178171551496749066625688813163831341757850536842806898
		8416699893193839489440896627862140602547385139200\,y^3\\
		&+100949514162612457570649681616394061405729312505617530262781683638
		53591624283949865960156104740370130615039688704000\,y^2\\
		&+65430476098993380035597129871559915990678672344452131348037447
		794353105436891172692500199659931839560250390321561600000\,y\\
		&-1360481782466194260782581568605011381837207631285341493876218186921892069999763325780780086952910422735909435113629286400000)\partial^{21} 
	\end{align*}
	
\end{landscape}

\subsection{General $0$-order term}
\noindent{$H_3(X)$ :}
\begin{align*}  
	( 290\,y-y^3 )
	+ ( 528\,y^2-1560)\partial
	+ ( 243\,y^3-1404\,y )\partial^{2}
	+ ( 27\,y^4-648\,y^2+2160 )\partial^{3}
\end{align*} 
\noindent{$H_4(X)$ :} 
\begin{align*}
	( -y^2+50\,y+24 )
	+ (64\,y^2+72\,y-1008 )\partial
	+ (16\,y^3-48\,y^2-576\,y+1728 )\partial^{2}
\end{align*} 
\noindent{$H_5 (X)$ :} 
\begin{align*}
	& (y^9-104800744\,y^7+174104044032\,y^5-82431615212544\,y^3+9617056740900864\,y) 
	\\& +(-83053520\,y^8+191761742080\,y^6-148596701936640\,y^4
	+33440484399022080\,y^2\\
	&-868706901405204480)\partial
	\\& +(-23029125\,y^9+72332912000\,y^7-88767223008000\,y^5
	+32039796049920000\,y^3\\
	&-1984593650909184000\,y)\partial^{2}
	\\& +(-2831875\,y^{10}+11857320000\,y^8-22211556000000\,y^6
	+11983543971840000\,y^4\\
	&-1826589574103040000\,y^2+54875902433034240000)\partial^{3}
	\\& +(-156250\,y^{11}+855800000\,y^9-2353387200000\,y^7
	+1868056934400000\,y^5\\
	&-530407371571200000\,y^3+36302379968102400000\,y)\partial^{4}
	\\ &+(-3125\,y^{12}+22000000\,y^{10}-85519200000\,y^8
	+99156326400000\,y^6\\
	&-65065321267200000\,y^4+19243712957644800000\,y^2-849260402284953600000)\partial^{5}  
\end{align*}
\noindent{$H_6 (X)$ :}
\begin{align*} &
	( 15303970800 y - 252586320 \, y^2 -6227803 \, y^3 + 599 \, y^4 ) 
	\\ &
	+  ( -6722792640000 -28723248000 \, y   +   30858084000 \, y^2   - 247410960 \, y^3  -6390132 \, y^4) \partial 
	& \\ & 
	+ (  -25152249600000     -8314215840000 \, y  +  29111400000 \, y^2 +  14157844200   \, y^3       
	\\  &
	-43020180\, y^4 -1746684\, y^5  ) \partial^2 
	& \\ & 
	+ (  1173771648000000   + 27946944000000 \, y     -3912572160000 \, y^2  -13197168000 \, y^3  
	\\   &
	+ 1633473000\, y^4     -129384 \, y^6 ) \partial^3
\end{align*} 
\begin{align*} 
	&  (109\,y^2-154953\,y )  +(-307656\,y^2-18871668\,y+144726480)\partial
	\\& + (-176580\,y^3-24578892\,y^2-161845560\,y+8329132800)\partial^{2}
	&  \\& + (-23544\,y^4-11122704\,y^3-573826680\,y^2+19960408800\,y+165815424000)\partial^{3}
	&  \\& + (-1635552\,y^4-45774720\,y^3+12374920800\,y^2
	-72643392000\,y-3945697920000)\partial^{4}
\end{align*} 

\newpage 

\begin{landscape}
	
	\noindent{$H_7(X)$ :}
	\begin{align*}
		& (y^{19}-507387619116342\,y^{17}-802740828170490338304\,y^{15}
		-8563913018807394706050032640\,y^{13}
		+6928877719389349036059564121989120\,y^{11}\\
		&-1889488355078466263270506243756474368000\,y^9
		+199783737934449191867909110955397754650624000\,y^7\\
		&
		-6209514567366557732151190998588725880170741760000\,y^5+60373523764591196745889692925201803791994716160000000\,y^3
		\\&-142617550691402319475078493723129063669971392921600000000\,y) 
		\\
		& +(-232921388527002\,y^{18}-455003916480962263296\,y^{16}
		-5595505479182703900889543680\,y^{14}
		+5359267163470096952878203878400000\,y^{12}\\
		&-1790750447596077728171307457048775884800\,y^{10}
		+242192518093896257570639060573438857445376000\,y^8\\
		&
		-10976522305118653612390351373153959652024647680000\,y^6
		+188801309522378030819713079326573875028316651520000000\,y^4\\
		&
		-849578054228104945769620510436241145254090532454400000000\,y^2
		+553541642524916989192828703979083182844939128012800000000000)\partial
		\\
		& +(-43131498107985\,y^{19}-101698009998146142336\,y^{17}
		-1441001761781953447350554112\,y^{15}
		+1622622157821339467404150681712640\,y^{13}\\
		&-655363295937332809553707487080390656000\,y^{11}
		+110222884262423760992059821959749682921472000\,y^9\\
		&
		-6963162577316880535809398372458440701132144640000\,y^7
		+195463077661033812028456925271451366713137823744000000\,y^5\\
		& -1537270398566555328673592257653757937696270765260800000000\,y^3
		+2834484949340410178499506262236587867216147717816320000000000\,y)\partial^{2}
		\\ 
		& +(-4190138977689\,y^{20}-11709548869474068480\,y^{18}
		-190740259095998425806343680\,y^{16}
		+250761545284409485936114645002240\,y^{14}\\
		&-121056389124490719167027426982862848000\,y^{12}
		+24786728314624304742443576061745973624832000\,y^{10}\\
		&
		-2106315036035646939339458645316146587458600960000\,y^8
		+89577998116601781174849293487939487260274065408000000\,y^6\\
		&  -1188601371786900616063301231279891037592677148262400000000\,y^4
		+4756060856343753414312957597082444762793487347220480000000
		000\,y^2\\
		&-3589642532973667755782665012000095549005659964964864000000000000)\partial^{3}
		\\
		& +(-231324085692\,y^{21}-754747910118973008\,y^{19}
		-14088550819983097003822080\,y^{17}
		+21476940651405739826066461593600\,y^{15}\\&-12274037662580531229209596454404915200\,y^{13}
		+3010279882178958725770347870540457574400000\,y^{11}\\&
		-334496746109646544834680912514027179343872000000\,y^9
		+20266422216760105834401804553844537058561884160000000\,y^7\\
		&-433438326566927102800531632421680585042528357580800000000\,y^5
		+3372408742538073162201990793367288516864324570972160000000000\,y^3\\
		&-7205789497682540844605222569506158549428259801530368000000000000\,y)\partial^{4}+
		\\
	\end{align*}
	
	\newpage
	
	\begin{align*} 
		& +(-7276708299\,y^{22}-27371665921506480\,y^{20}
		-583359969825161143856640\,y^{18}
		+1024344767123350050915641825280\,y^{16}\\
		&-687170840461794097098013847693721600\,y^{14}
		+199322251667829117690419038929864818688000\,y^{12}\\
		&-28307494810633283329218161045332817844633600000\,y^{10}
		+2327426796576918553408199766888305080789893120000000\,y^8\\
		&-76446963841114494711288612434429137708914416025600000000\,y^6
		+1188001939451026076951506842982872478864827450654720000000000\,y^4\\
		&-603385577917814332915349953862104444107717849435340800000000
		0000\,y^2\\
		&+6607922436713149268288365697555293316398623854533017600000000000000)\partial^{5}  \\
		& +(-121060821\,y^{23}-519363136879776\,y^{21}
		-12589329361889934581760\,y^{19}
		+25302126251373511739591669760\,y^{17}\\
		&-19772202464884839592647254251929600\,y^{15}
		+6714017565928228688520035728081354752000\,y^{13}\\
		&-1195934597522197165985191951618773142732800000\,y^{11}
		+128496817550449617359186264295963354778828800000000\,y^9\\
		&-6220725820796017230227466495015851143758191001600000000\,y^7
		+180826623964990243551570923667810203111844678205440000000000\,y^5\\
		&-1823886381111265337505606804877800162242682358333440000000000
		000\,y^3\\
		&+4326605465487310287646936158092962227193444231544832000000000000000\,y)\partial^{6}
		\\
		& +(-823543\,y^{24}-3992101633296\,y^{22}-109632800412792244224\,y^{20}
		+250682421106487697950730240\,y^{18}
		-226599508735853638112124855091200\,y^{16}\\
		&
		+89319208475424581112176973409222656000\,y^{14}
		-19639979563871778141201022209181011148800000\,y^{12}\\
		&
		+2678638244295631856999285390534769912053760000000\,y^{10}
		-185111225697649160914098420368554727334897254400000000\,y^8\\
		&
		+9684072297784585284993126005710067503579036385280000000000\,y^6
		-353852008232963377959228002588661348316332184043520000000000000\,y^4\\
		&
		+2802757980193056097105341381472739246776766705958912000000000
		000000\,y^2\\
		&-388968213288975237753969849076520310923854301412458496000000000
		0000000)\partial^{7}
	\end{align*}
\end{landscape}

\section{Ring operations and Stein operators}\label{apring}

Our algorithm can also be applied to general univariate Gaussian polynomials. We collect some examples here, in which we consider linear combinations of Gaussian Hermite polynomials. In addition to this proof of concept, we are able to observe from our results that a small change in the target distribution can lead to substantial differences in the complexity of the Stein operators; see Remark \ref{remringop}.

\vspace{3mm}

\noindent{(a) $Y=H_1(X)+H_2(X)$   :}
\begin{align*} 
	y+  (- 4y - 3)\partial	+(4y + 5)\partial^2  
\end{align*}
\noindent{(b) $Y=H_1(X)+H_2(X)+H_3(X)$   :}
\begin{align*} &
	y + (- 4y - 9)\partial +  (- 92y - 43)\partial^2 
	+  (- 27y^2 + 82y + 119)\partial^3  
	+ (27y^2 + 392y + 49)\partial^4    \\ &
	+ (378y^2 + 196y - 686)\partial^5 
\end{align*}
\noindent{(c)  $Y=H_2(X) + H_3(X)$     : }
\begin{align*} &
	134 y +  (- 81y^2 - 172y - 424)\partial  
	+ (243y^2 + 6276y + 1056)\partial^2   \\ &
	+ (9504y^2 + 1292y - 40296)\partial^3 + (2187y^3 + 2214y^2 - 12364y - 13912)\partial^4 
\end{align*}
\begin{align*} &
	y + (- 4y - 8)\partial + (- 98y - 26)\partial^2 
	+ (- 27y^2 + 118y + 324)\partial^3  
	+ (27y^2 + 536y - 188)\partial^4  \\ &
	+(540y^2 - 80y - 2960)\partial^5    
\end{align*}
\noindent{(d)  $Y=H_1(X)+H_2(X)+H_3(X)+H_4(X) $       : }
\begin{align*} &
	8y + (- 633y - 264)\partial +  (- 256y^2 + 17392y + 16033)\partial^2 
	+  (16928y^2 - 49627y - 233513)\partial^3   \\ &
	+ (2048y^3 - 215304y^2 - 707732y + 1361327)\partial^4 & \\ &
	+  (- 45312y^3 + 156709y^2 - 408426y - 1868559)\partial^5 & \\ &
	+  (220928y^3 + 8481141y^2 + 37742788y - 15880534)\partial^6    & \\ &
	+ (2062080y^3 - 2592195y^2 - 95069510y - 125583700)\partial^7  & \\ &
	+ (- 12613120y^3 - 99870290y^2 - 29364920y + 678349360)\partial^8 
\end{align*}

\noindent{(e)  $Y=H_1(X)+H_4(X)$     :}
\begin{align*} &
	y+ (- 10y - 25)\partial +  (- 32y^2 - 600y + 186)\partial^2 + (192y^2 + 17706y + 12888)\partial^3   \\ &
	+ (256y^3 + 45312y^2 + 346032y - 486783)\partial^4  & \\ &
	+   (7680y^3 - 362304y^2 - 2741472y + 1001322)\partial^5  & \\ &
	+  (- 129024y^3 + 145152y^2 + 8273664y + 11580408)\partial^6  & \\ &
	+ (870912y^3 + 7838208y^2 - 2939328y - 88087986)\partial^7 
\end{align*}

\noindent{(f)  $Y=H_2(X)+H_4(X)$ :}
\begin{align*} 
	y + (- 42y - 26)\partial +  (- 16y^2 + 124y + 316)\partial^2 + 
	(160y^2 + 360y - 1360)\partial^3 
\end{align*}

\noindent{(g)   $Y=H_3(X) + H_4 (X)$    : }
\begin{align*} &
	8y +  (- 649y - 240)\partial +  (- 256y^2 + 18670y + 17646)\partial^2  + (17440y^2 - 66183y - 441024)\partial^3  \\ &
	+ (2048y^3 - 258888y^2 - 759228y + 4060134)\partial^4  & \\ &
	+ (- 49408y^3 + 1032645y^2 + 6790131y - 12082662)\partial^5  & \\ &
	+      (460032y^3 + 6676839y^2 + 13811904y - 80220780)\partial^6      & \\ &
	+ (576000y^3 - 473202y^2 - 43174782y - 90319212)\partial^7  & \\ &
	+ (- 4243968y^3 - 43417782y^2 - 1790424y + 658876032)\partial^8   
\end{align*}

\begin{rem}\label{remringop}
	Stein's method is rather sensitive to small perturbations in the target. Even a modest change in the target random variable $Y$ can drastically change the original polynomial Stein operators in terms of the order (the length of null control sequence) as well as the maximum degree of the polynomial coefficients. For example, consider the target $Y=-6H_2(X) + H_4(X)= X^4 -3$ that is obtained from the target random variable in item (f) via a multiplication by $-6$ of its first summand. For $Y$ our code returns the following algebraic polynomial Stein operator of order two
	\begin{align*}
		y + (- 32y - 96)\partial  + (- 16y^2 - 96y - 144)\partial^2.
	\end{align*}
	A similar example of this type is the target $Y=H_3(X)+3H_1(X) =X^3$ with algebraic polynomial Stein operator (compare with \eqref{h3x-new} and \eqref{h3x})
	\begin{align*}
		y - 15 \partial - 81 y \partial^2 - 27y^2 \partial^3.
	\end{align*}
	In contrast, consider the target variable $Y=H_4(X)+H_1(X)$ in item (e) that is perturbed by a copy of the underlying standard Gaussian random variable $X$. One can observe a drastic change in the complexity of the original algebraic polynomial Stein operator \eqref{h4x} not only in the order but also in the maximum degree of the polynomial coefficients.
\end{rem}


\begin{thebibliography}{9}
\bibitem{am07} \textsc{Antsaklis, P. J., Michel, A. N.} (2007). \emph{A linear systems primer.} Birkh\"{a}user
Boston, Inc., Boston, MA.

\bibitem{aaps19a}
\textsc{Arras, B., Azmoodeh, E., Poly, G., Swan, Y.} (2019).
\newblock A bound on the 2-Wasserstein distance between linear combinations of independent random variables. 
\newblock \emph{Stochastic Processes and their Applications} $\mathbf{129}$ pp. 2341--2375.

 \bibitem{aaps19b}  
\textsc{Arras, B., Azmoodeh, E., Poly, G., Swan, Y.} (2020).
 \newblock Stein characterizations for linear combinations of gamma random
variables. 
\newblock \emph{Brazilian Journal of Probability and Statistics} $\mathbf{34}$, pp. 394--413.

\bibitem{ah19} 
\textsc{Arras, B., Houdr\'e, C.} (2019).
\newblock \emph{On Stein's method for infinitely divisible laws with finite first
moment.}
\newblock Springer Briefs.

\bibitem{a-m-p-s}
\textsc{Arras, A.,  Mijoule, G.,  Poly, G.,  Swan, Y.} (2017).
\newblock  A new approach to the Stein-Tikhomirov method: with applications to the second Wiener chaos and Dickman convergence.
\newblock \texttt{https://arxiv.org/abs/1605.06819}.

\bibitem{a-c-p}
\textsc{Azmoodeh, E. Campese, S., Poly G.} (2014).
\newblock Fourth {M}oment {T}heorems for {M}arkov diffusion generators.
\newblock {\em Journal of Functional Analysis} $\mathbf{266}$ pp. 2341--2359.

\bibitem{a-g-tetilla}
\textsc{Azmoodeh, E., Gasbarra, D.} (2018).
\newblock New moments criteria for convergence towards normal product/tetilla laws. \newblock \texttt{https://arxiv.org/abs/1708.07681}. 

\bibitem{a-g-g21}
\textsc{Azmoodeh, A., Gasbarra, D., Gaunt. R. E.} (2021).
\newblock An asymptotic approach to proving sufficiency of Stein characterisations. 
\newblock \texttt{https://arxiv.org/abs/2109.08579}.

\bibitem{azmooden} 
\textsc{Azmoodeh, E., Peccati, G., Poly, G.} (2015).
\newblock  Convergence towards linear combinations of chi-squared random variables: a Malliavin-based approach.
\newblock  \emph{S\'{e}minaire de Probabilit\'{e}s} XLVII (special volume in memory of Marc Yor)  pp. 339--367.

\bibitem{ap18} 
\textsc{Azmoodeh, E., Peccati, G., Yang, X.} (2018).
\newblock Malliavin-Stein Method: a Survey of Recent Developments. \newblock \emph{Modern Stoch. Theory Appl.} $\mathbf{8}$(2021), 141--177.


\bibitem{baker}
\textsc{Baker, A.} (1975).
 \textit{Transcendental Number Theory.} Cambridge University Press, First edition.

\bibitem{barbour2} 
\textsc{Barbour, A. D.} (1990).
\newblock Stein's method for diffusion approximations.
\newblock  \emph{Probability Theory and Related Fields} $\mathbf{84}$  pp. 297--322.



\bibitem{henri-1}
\textsc{Bourl\'es, H.} (2010).
\newblock \emph{Linear Systems}.
\newblock  Wiley-series : Control Systems, Robotics and Manufacturing; 1 edition.


\bibitem{henri-2}
\textsc{Bourl\'es, H.,  Marinescu, B. } (2011).
\newblock \emph{Linear time-varying systems. Algebraic-analytic approach}. 
\newblock Lecture Notes in Control and Information Sciences, 410. Springer-Verlag, Berlin.


\bibitem{chen} 
\textsc{Chen, L. H. Y., Goldstein, L., Shao, Q.--M.} (2011).
\newblock  \emph{Normal Approximation by Stein's Method.}
\newblock Springer.

\bibitem{dz}
\textsc{Diaconis, P., Zabell, S.} (1991).
\newblock Closed Form Summation for Classical Distributions: Variations on a Theme of De Moivre. 
\newblock \emph{Statistical Science} $\mathbf{6}$ pp. 284--302.

\bibitem{dobler beta}
\textsc{ D\"{o}bler, C.} (2015).
\newblock Stein's method of exchangeable
  pairs for the beta distribution and
  generalizations.
  \newblock \emph{Electronic Journal of Probability} $\mathbf{20}$ no$.$ 109 
  pp. 1--34. 

\bibitem{dgv} 
\textsc{D\"{o}bler, C., Gaunt, R. E., Vollmer, S. J.} (2017).
\newblock  An
  iterative technique for bounding derivatives of solutions of Stein
  equations.
  \newblock \emph{Electronic Journal of Probability} $\mathbf{22}$ no$.$ 96  pp. 1--39.
  
  \bibitem{dp16} 
  \textsc{D\"{o}bler, C., Peccati, G.} (2018).
\newblock   The Gamma Stein equation and non-central de Jong theorems. \newblock \emph{Bernoulli} $\mathbf{24}$  pp. 3384--3421.

\bibitem{ev13}
\textsc{ Eden, R., Viens, F.} (2013). 
\newblock General upper and lower tail estimates using Malliavin calculus and Stein's equations. 
\newblock In \emph{Seminar on Stochastic Analysis, Random Fields and Applications VII} (R.C. Dalang, M. Dozzi and F. Russo, eds.) \emph{Progress in Probability} $\mathbf{67}$ Birkh\"{a}user/Springer,
Basel, pp. 55--84.

\bibitem{ev15}
\textsc{Eden, R., V\'iquez, J.} (2015). 
\newblock Nourdin-Peccati analysis on Wiener and Wiener-Poisson space for general distributions. 
\newblock \emph{Stochastic Processes and their Applications} $\mathbf{125}$ pp. 182--216.

	\bibitem{e-t}
\textsc{Eichelsbacher, P., Th\"ale, C.} (2014).
\newblock Malliavin-Stein method for Variance-Gamma approximation on Wiener space.
\newblock \emph{Electronic Journal of Probability} $\mathbf{20}$ no. 123 pp. 1--28.

\bibitem{Fuhrmann}
\textsc{Fuhrmann, P.  A.} (1981).
\newblock \emph{Linear systems and operators in Hilbert space}. 
\newblock McGraw-Hill International Book Co., New York.


\bibitem{fuhrmann2015} 
\textsc{Fuhrmann, P. A., Helmke, U.} (2015).
\newblock \emph{The Mathematics of Networks of Linear Systems}.
\newblock Universitext Springer  International Publishing.

\bibitem{gaunt vg} 
\textsc{Gaunt, R. E.} (2014).
\newblock Variance-Gamma approximation via Stein's Method.  
\newblock \emph{Electronic Journal of Probability} $\mathbf{19}$ no. 38 pp. 1--33.

\bibitem{gaunt-pn} 
\textsc{Gaunt, R. E.} (2017).
\newblock On Stein's method for products of normal random variables and zero bias couplings. 
\newblock \emph{Bernoulli} $\mathbf{23}$ pp. 3311--3345.

\bibitem{gaunt-ngb} 
\textsc{Gaunt, R. E.} (2018). 
\newblock Products of normal, beta and gamma random variables: Stein operators and distributional theory.  
\newblock\emph{Brazilian Journal of Probability and Statistics} $\mathbf{32}$  pp. 437--466.

\bibitem{gaunt-34} 
\textsc{Gaunt, R. E.} (2019).
\newblock Stein operators for variables form the third and fourth Wiener chaoses. 
\newblock \emph{Statistics and Probability Letters} $\mathbf{145}$ pp. 118--126.

\bibitem{gms19}
\textsc{Gaunt, R. E., Mijoule, G., Swan, Y.} (2019).
\newblock  An algebra of Stein operators. 
\newblock \emph{Journal of Mathematical Analysis and Applications} $\mathbf{469}$ pp. 260--279.

\bibitem{gms19b}
\textsc{Gaunt, R. E., Mijoule, G., Swan, Y.} (2020).
\newblock  Some new Stein operators for product distributions. 
\newblock \emph{Brazilian Journal of Probability and Statistics}. $\mathbf{34}$ pp. 795--808.

\bibitem{goldstein4} 
\textsc{Goldstein, L., Reinert, G.} (2013).
  Stein's method for
  the Beta distribution and the P\'{o}lya-Eggenberger Urn.
  \emph{Journal of Applied Probability} $\mathbf{50}$  pp. 1187--1205.
  
\bibitem{gotze}  
\textsc{G\"{o}tze, F.} (1991).
\newblock On the rate of convergence in the multivariate CLT. 
\newblock \emph{Annals of Probability} $\mathbf{19}$ pp. 724--739.

\bibitem{ind61} 
\textsc{Indritz, J.} (1961).
\newblock An Inequality for Hermite Polynomials. 
\newblock \emph{Proceedings of the American Mathematical Society} $\mathbf{12}$, pp. 981--983.

\bibitem{k80}  \textsc{Kailath, T.} (1980). \emph{Linear Systems.} Prentice-Hall, Englewood Cliffs, NJ.

 \bibitem{Kalman60}
\textsc{Kalman, R. E.} (1960) On the general theory of control systems. In: \textit{Proceedings of the First
World Congress of the International Federation of Automatic Control, 1960}, pp. 481--493.

\bibitem{Kalman}
\textsc{Kalman, R. E.} (1963). 
\newblock Mathematical Description of Linear Dynamical Systems.
\newblock  \emph{SIAM Journal on Control} $\mathbf{1}$ pp. 152--192.

\bibitem{kawano}
\textsc{Kawano, Y., Ohtsuka, T.} (2016).
\newblock Commutative algebraic methods for controllability of discrete-time polynomial systems.
\newblock\emph{International Journal of Control}  $\mathbf{89}$ pp. 343--351.

\bibitem{k93}
\textsc{Klamka, J.} (1993). Controllability of dynamical systems. A survey.
 \textit{Arch Control Sci},
\textbf{2} pp. 281--307.

\bibitem{k-b-Ustatistics}
\textsc{Koroljuk, V. S., Borovskich, Yu. V.} (1994).
\newblock 	\emph{Theory of U-statistics}. 
\newblock Mathematics and its Applications, 273. Kluwer Academic Publishers Group, Dordrecht.

\bibitem{kusuotud} 
\textsc{Kusuoka, S.,  Tudor, C. A.} (2012).
\newblock Stein's method for
invariant measures of diffusions via Malliavin calculus.
\newblock \emph{Stochastic processes and their Applications} $\mathbf{122}$  pp. 1627--1651.

\bibitem{lee-Ustatistics}
\textsc{Lee, A. J.} (1990).
\newblock \emph{U-statistics. Theory and practice}.
\newblock  Statistics: Textbooks and Monographs, 110. Marcel Dekker, Inc., New York.

\bibitem{ley} 
\textsc{Ley, C., Reinert, G., Swan, Y.} (2017).
\newblock  Stein's method for comparison of univariate distributions.  \newblock \emph{Probability Surveys} $\mathbf{14}$  pp. 1--52.

\bibitem{ls13}
\textsc{Ley, C., Swan, Y.} (2013).
\newblock  Stein's density approach and information inequalities. \newblock  \emph{Electronic Communications in Probability} $\mathbf{18}$ no$.$ 7, pp. 1--14.

\bibitem{luk} 
\textsc{Luk, H.} (1994).
\newblock  \emph{Stein's Method for the Gamma Distribution and Related Statistical Applications.}  
\newblock PhD thesis, University of Southern California, 

\bibitem{n-n-p-mixedGaussian}
\textsc{Nourdin, I., Nualart, D.,  Peccati, G.} (2016). 
\newblock Quantitative stable limit theorems on the Wiener space. 
\newblock \emph{Annals of Probability} $\mathbf{44}$ pp. 1--41.

\bibitem{np09} 
\textsc{Nourdin, I., Peccati, G.} (2009).
\newblock Stein's method on Wiener chaos. 
\newblock \emph{Probability Theory and Related Fields} $\mathbf{145}$  pp. 75--118.

\bibitem{np10}
\textsc{Nourdin, I., Peccati, G.} (2010).
\newblock   Cumulants on the Wiener space.
\newblock \emph{Journal of Functional Analysis} $\mathbf{46}$  pp. 45--58.

\bibitem{n-p-book}
\textsc{Nourdin, I., Peccati, G.} (2012).
\newblock \emph{Normal Approximations Using Malliavin Calculus: from Stein's Method to Universality}.
\newblock  Cambridge Tracts in Mathematics. Cambridge University Press.

\bibitem{np15}
\textsc{Nourdin, I., Peccati, G.} (2015).
\newblock The optimal fourth moment theorem. 
\newblock \emph{Proceedings of the American Mathematical Society} $\mathbf{143}$ pp. 3123--3133.

\bibitem{npr10} 
\textsc{Nourdin, I., Peccati, G., R\'{e}veillac, A.} (2010).
\newblock  Multivariate normal approximation using Stein's method and Malliavin calculus. 
\newblock \emph{ Annales de l'Institut Henri Poincar\'{e} (B) Probabilit\'{e}s et Statistiques} $\mathbf{46}$  pp. 45--58.

\bibitem{npoly12} 
\textsc{Nourdin, I., Poly, G.} (2012).
\newblock Convergence in law in the second Wiener/Wigner chaos.
  \newblock \emph{Electronic Communications in Probability} $\mathbf{17}$ no. 36  pp. 1--12.

\bibitem{n-n-book}
\textsc{Nualart, D., Nualart, E.} (2018).
\newblock \emph{Introduction to Malliavin Calculus}.
\newblock  Institute of Mathematical Statistics Textbooks. Cambridge University Press.

\bibitem{peccati14} 
\textsc{Peccati, G.} (2014).
\newblock  Quantitative CLTs on a Gaussian space: a survey of recent developments. 
\newblock \emph{ESAIM: Proceedings} $\mathbf{44}$  pp. 61--78.

\bibitem{pekoz} 
\textsc{Pek\"oz, E., R\"ollin, A., Ross, N.} (2013).
\newblock Degree asymptotics with rates for preferential attachment random graphs.
\newblock \emph{Annals of Applied Probability} $\mathbf{23}$ pp. 1188--1218.

\bibitem{pike} 
\textsc{Pike, J., Ren, H.} (2014).
\newblock Stein's method and the Laplace distribution.
\newblock \emph{ALEA: Latin American Journal of Probability and Mathematical Statistics} $\mathbf{11}$  pp. 571--587.

\bibitem{ross} 
\textsc{Ross, N.} (2011). 
\newblock Fundamentals of Stein's method.  
\newblock \emph{Probability Surveys} $\mathbf{8}$ pp. 210--293.

\bibitem{Rugh}
\textsc{Rugh, W.} (1993).
\newblock	\emph{Linear system theory}.
\newblock	Prentice Hall Information and System Sciences Series. 

\bibitem{schoutens} 
\textsc{Schoutens, W.} (2001).
\newblock  Orthogonal polynomials in Stein's  method. 
\newblock \emph{Journal of Mathematical Analysis and Applications} $\mathbf{253}$  pp. 515--531.



\bibitem{stein} 
\textsc{Stein, C.} (1972). 
\newblock A bound for the error in the normal approximation to the the distribution of a sum of dependent random variables.
\newblock  In \emph{Proceedings of the Sixth Berkeley Symposiumon Mathematical Statistics and Probability}  vol. 2, University of California Press, Berkeley,  pp. 583--602.

\bibitem{stein2}
\textsc{Stein, C.} (1986).  
\newblock \emph{Approximate Computation of Expectations.}
\newblock IMS, Hayward, California.

\bibitem{stein3}
\textsc{Stein, C., Diaconis, P., Holmes, S., Reinert G.} (2004). \newblock Use of exchangeable pairs in the analysis of simulations. 
\newblock In \emph{Stein's method: expository lectures and applications} (P. Diaconis and S. Holmes, eds.), IMS Lecture Notes Monograph Series vol. 46, Beachwood, Ohio, USA: Institute of Mathematical Statistics, pp. 1--26.

\bibitem{Trig1}
\textsc{Triggiani, R. } (1975). 
\newblock Controllability and Observability in Banach Spaces with Bounded Operators.
\newblock \emph{SIAM Journal on Control and Optimization} $\mathbf{13}$ pp. 462--491.

\bibitem{Trig2}
\textsc{Triggiani, R. } (1976).
\newblock Extensions of Rank Conditions for Controllability and Observability to Banach Space and Unbounded Operators.
\newblock \emph{SIAM Journal on Control and Optimization} $\mathbf{14}$ pp. 313--334.

\bibitem{Liu} 
\textsc{Xu, L.} (2019). 
\newblock Approximation of stable law in Wasserstein-1 distance by
  Stein's method.
\newblock  \emph{Annals of Applied Probability} $\mathbf{29}$ pp. 458--504.


\end{thebibliography}
\end{document}